\documentclass[11pt]{article}
\usepackage{fullpage}
\usepackage[psamsfonts]{eucal}
\usepackage{mathptmx}
\usepackage{microtype}
\renewcommand{\title}[1]{ 
  \thispagestyle{plain} 
  \vspace*{-14pt} 
  \vskip 79pt
  {\centering
   \fontsize{15}{17}\bfseries\boldmath\selectfont
   #1\par} 
  \vskip 1em
}
\renewcommand{\author}[1]{ 
  \par
  {\centering\normalsize
   #1\par\vspace*{0.05in}} 
}
\newcommand{\address}[1]{ 
  \par
  {\centering
   \fontsize{8}{10}\itshape
   #1\par} 
  \par\vspace*{11pt} 
}
\usepackage{fullpage}
\usepackage[psamsfonts]{eucal}
\usepackage{amsmath}
\usepackage{amsfonts}
\usepackage{amssymb}
\usepackage{amstext}
\usepackage[symbol]{footmisc}
\usepackage[noadjust]{cite}
\usepackage{amscd}
\usepackage{amsthm}
\usepackage{makeidx}
\usepackage{tikz}
\usepackage{graphicx}
\usepackage[colorlinks=true,linkcolor=blue]{hyperref}
\usepackage{centernot}  

\usetikzlibrary{calc}

\usepackage{supertabular}
\usepackage{enumitem}
\newcommand{\re}{\mathrm{e}}
\usepackage{titling}
\usepackage{wasysym}
\usepackage{comment}
\usepackage{cleveref}
\usepackage{listings}
\usepackage{relsize}

\newcommand{\RR}{\mathbb{R}}
\newcommand{\EE}{\mathbb{E}}

\newcommand{\ov}{\mathrm{ov}}

\newcommand{\bZ}{\mathbb{Z}}

\newcommand{\CC}{\mathbb{C}}

\usepackage{microtype}
\newcommand{\E}{\mathbb{E}}

\theoremstyle{plain}
\newtheorem{thm}{Theorem}[section]
\newtheorem{remark}[thm]{Remark}

\newtheorem{prop}[thm]{Proposition}
\newtheorem{lemma}[thm]{Lemma}
\newtheorem{lem}[thm]{Lemma}
\newtheorem{obs}[thm]{Observation}

\newtheorem{dfn}[thm]{Definition}

\newtheorem{cor}[thm]{Corollary}

\newcommand{\wh}{\widehat}

\def\title#1{
    \thispagestyle{plain}
    \vspace*{-14pt}
    \vskip 39pt
    {\centering{\titlefont #1\par}} 
    \vspace*{4pt}
    \vskip 1em
}
\def\author#1{\par
    {\centering{\authorfont#1}\par\vspace*{0.05in}}
}
\def\address#1{\par
    {\centering{\affiliationfont#1\par}}\par\vspace*{11pt}
}
\crefname{thm}{theorem}{theorems}
\Crefname{thm}{Theorem}{Theorems}
\Crefname{prop}{Proposition}{Propositions}
\Crefname{lem}{Lemma}{Lemmas}
\Crefname{lemma}{Lemma}{Lemmas}
\Crefname{dfn}{Definition}{Definitions}

\def\titlefont{\fontsize{14}{21}\boldmath\selectfont\centering{}}
\def\authorfont{\fontsize{13}{15}}

\let\affiliationfont\rhfont

\renewcommand{\E}{\mathbb{E}}

\DeclareMathOperator{\dist}{dist}
\DeclareMathOperator{\SBM}{SBM}

\DeclareMathOperator{\diag}{diag}
\DeclareMathOperator{\fin}{fin}

\definecolor{forestgreen}{RGB}{34, 139, 34}

\newcommand{\deq}{\mathrel{\mathop:}=}

\newcommand{\cL}{\mathcal{L}}

\newcommand{\be}{\begin{equation}}
\newcommand{\ee}{\end{equation}}

\newenvironment{proof*}[1][\proofname]{
  
  \begin{proof}[#1]}{\end{proof}}

\numberwithin{equation}{section}

\newcommand{\Poi}{\mathrm{Poi}}
\newcommand{\ch}{\mathrm{ch}}

\begin{document}

\title{Bethe Hessian Recovery through the Kesten–Stigum Regime}

\noindent
\begin{minipage}[c]{0.6\textwidth}
  \author{Dingding Dong}
  \address{California Institute of Technology\\
    ddong124@caltech.edu}
\end{minipage}
\begin{minipage}[c]{0.2\textwidth}
  \author{Theo McKenzie}
  \address{Yale University\\
    theo.mckenzie@yale.edu}
\end{minipage}

\begin{abstract}
We show that outlier eigenvectors of the Bethe Hessian achieve weak recovery in the stochastic block model down to the Kesten–Stigum threshold, for every fixed common expected degree $d>1$. This extends previous recovery guarantees, which required $d$ to be sufficiently large. This threshold is information-theoretically optimal for the
symmetric two-community model and is conjectured to mark
the limit of efficient recovery in broader symmetric models. Our result matches
the recovery guarantees of non-backtracking spectral methods
while using Hermitian $n\times n$ matrices.

To do this, we identify recursive formulas for the negative outlier
eigenvalues and show that their eigenspaces have nonvanishing
overlap with the corresponding planted eigenspaces.
The proof constructs approximate eigenvectors from the planted
labels on the limiting Poisson tree, with weights given by its
Green's function, and transfers these approximations to the
finite graph. 
\end{abstract}
\section{Introduction}

The stochastic block model is a canonical test case for partitioning algorithms. Vertices are assigned planted community labels, and edges are generated independently with probabilities based on the relationship between communities \cite{holland1983stochastic}. Studying recovery of the different communities has a rich history, with applications in many
areas, including social network analysis,
computational biology, and neuroscience
\cite{holland1983stochastic, airoldi2008mixed, faskowitz2018weighted}.

In the sparse regime, a central question has been when the planted communities can be detected at fixed expected degree. Here we focus on \emph{weak recovery}, that is, having a strictly non-negligible advantage over random guessing. Exact recovery is information-theoretically impossible, as, for example, some vertices are isolated. A natural benchmark is given by the \emph{Kesten--Stigum threshold} (see \Cref{def:KS}), 
above which linear estimators using the planted labels on distant
generations of the limiting tree retain nonvanishing correlation
with a label statistic at the root \cite{mossel2015reconstruction}.
This threshold is information-theoretically sharp for weak recovery
in the symmetric two-community model 
\cite{mossel2015reconstruction,massoulie2014community,mossel2018proof}, and it is conjectured to be tight for efficient algorithms for any fixed number of symmetric communities \cite{decelle2011asymptotic, mossel2015reconstruction, abbe2018community}, although information-theoretic recovery can be possible
below it \cite{abbe2018proof,chin2025stochastic}. 

Following predictions from statistical physics
\cite{decelle2011asymptotic}, efficient algorithms for weak recovery were developed
that reach the Kesten--Stigum threshold in the symmetric
two-community model
\cite{massoulie2014community,mossel2018proof}.
 In parallel, the idea of
\emph{spectral redemption} showed how the non-backtracking operator
could overcome the limitations of standard spectral methods in
sparse graphs, and planted communities could be recovered from its eigenvectors \cite{krzakala2013spectral}. Its ability to reach the
Kesten--Stigum threshold was subsequently established rigorously
\cite{bordenave2018nonbacktracking}.

The non-backtracking operator is non-Hermitian, however, and is
indexed by directed edges rather than vertices, meaning it can be much larger. These features
motivated the study of a smaller Hermitian alternative, the Bethe Hessian,
coming from statistical physics \cite{saade2014spectral}.

For a graph $G$ with adjacency matrix $A$ and degree matrix $D$, the Bethe Hessian is
\[
H(t,G)\deq (t^2-1)I+D-tA.
\]
From a physical perspective, the Bethe Hessian describes the change in the Bethe free energy when we add a small bias to the random variables.  Negative eigenvalues therefore give directions in which adding a bias decreases the Bethe free energy \cite{saade2014spectral}. Moreover, the Ihara--Bass formula gives a relationship between the determinants of the non-backtracking operator and the Bethe Hessian \cite{bass1992ihara}.

Early numerical experiments suggested that the negative eigenvalues
and corresponding eigenvectors of the Bethe Hessian reveal the latent
structure of the stochastic block model \cite{saade2014spectral}.
Rigorous analysis subsequently established that the negative outliers
count the informative directions when the common expected degree
satisfies $d\geq 2$, and that their eigenvectors yield weak recovery
for sufficiently large $d$ \cite{Stephan2025community}.
The outlier count was later extended to every fixed $d>1$, while
weak recovery from the corresponding eigenvectors remained open
throughout this full range \cite{dong2026bethe}. Bethe Hessian methods have also been used to obtain \emph{robust}
recovery above the Kesten--Stigum threshold
\cite{mohanty2024robust}, through a method combining degree
truncation with a ``boosted'' version of the Bethe Hessian given by multiplying the Bethe Hessian on both sides by a matrix counting non-backtracking
walks, and additional subspace processing. In comparison, our result
concerns the informative eigenspaces of the
matrices $H(\sqrt d,G)$ and $H(-\sqrt d,G)$ themselves.

In this work, we identify the Bethe Hessian's limiting outlier eigenvalue locations and prove
that their eigenvectors retain the planted signal. We describe the informative eigenvectors and use them to prove weak recovery throughout the fixed-degree
regime. Consider a stochastic block model with fixed parameters,
common expected degree $d>1$, and mean offspring matrix $Q$.
A planted eigenvalue $\mu$ is called informative when $|\mu|>\sqrt d$;
positive and negative informative directions correspond to
$H(\sqrt d,G)$ and $H(-\sqrt d,G)$, respectively. The following is an informal version of what we are able to show; the full versions are given as \Cref{thm:eigenvectors} and \Cref{thm:weak-recovery}.

\begin{thm}[Informal]
For every fixed $d>1$, informative Bethe Hessian eigenspaces
have nonvanishing overlap with the corresponding planted
eigenspaces. Whenever a nonconstant planted direction is
informative, a randomized rounding of the informative
spectral coordinates achieves weak recovery.
\end{thm}

The construction also gives a recursive formula for the limiting outlier locations.
This closes the gap between counting informative outliers and using
their eigenvectors for recovery, showing that a spectral algorithm
based on $n\times n$ Hermitian matrices reaches the Kesten--Stigum
threshold without requiring the expected degree to tend to infinity. 
Weak recovery is only information theoretically possible when $d>1$ \cite{mossel2015reconstruction, mossel2024weak}, as otherwise the local limit Poisson tree is finite almost surely.

The main difficulty in passing to the sparse regime is that large fluctuations cause the eigenvector to have highly irregular structure. Previous works found well-behaved test vectors that identified negative quadratic forms, but still relied on some level of local approximation \cite{Stephan2025community,dong2026bethe}. Such approximations were too rough to guarantee the structure of true eigenvectors.   We instead construct
approximate eigenvectors with vanishing residual and
nonvanishing planted correlation. Together with the
outlier count, these properties identify the informative
eigenspaces throughout the fixed-degree regime.

The key point is that a planted direction above the Kesten--Stigum threshold gives rise to a negative outlier of the finite-graph Bethe Hessian. Therefore, we are looking for a solution to
\[
(H(t,G)+\delta I)\phi=0,\qquad \delta>0.
\]
Because this outlier is created by the \emph{global} planted structure, it is absent from the spectrum of the operator on the \emph{local} limit. Rather than inverting the singular finite-graph matrix at this eigenvalue, we create weights using the well-defined limiting resolvent $(H(t,T)+\delta I)^{-1}$. For the equal-expected-degree model considered here, the unmarked local limit $T$ is a Poisson Galton--Watson tree with mean offspring $d$.

The second insight is that these Green's function weights, together with the planted label field on the tree, can be used to identify the informative eigenspace. The tree resolvent supplies weights adapted to the local geometry, while the label field supplies the planted direction. Iterating this construction gives approximations that are consistent across neighboring vertices and satisfy the eigenvector equation in the limit.

To explain the weights, suppose $\phi$ is an eigenvector of $H(t,G)$ with eigenvalue $-\delta$. Then, for every $u\in V(G)$,
\[
(t^2-1+\deg(u)+\delta)\phi(u)
=t\sum_{w\sim u}\phi(w).
\]
For every $v\sim u$, whenever the relevant entries are nonzero, writing formally $\Gamma_u^{(v)}=\phi(u)/\phi(v)$ gives
\begin{equation}\label{eq:Schur}
\Gamma_u^{(v)}
=\frac{1}{t+t^{-1}(\deg(u)-1+\delta)
-\sum_{\substack{w\sim u\\w\neq v}}\Gamma_w^{(u)}}.
\end{equation}

We then perform this same recursion on the infinite local limit. There, the weights are defined by the  \emph{Green's function} rather than by ratios of a pre-existing eigenvector:
\[
\Gamma_u^{(v)}
=t\left\langle \mathbf e_u,
\bigl(H(t,T_u^{(v)})+\delta I\bigr)^{-1}\mathbf e_u
\right\rangle,
\]
where $T_u^{(v)}$ is the branch containing $u$ after removing the edge $\{u,v\}$. Here $H(t,T_u^{(v)})$ denotes the restriction of $H(t,T)$ to that branch, so its diagonal retains the degrees in $T$. With this convention, \eqref{eq:Schur} is exactly the Schur complement formula. Starting with $\Gamma_u^{(v)}=0$ at a finite-depth boundary and recursively updating gives finite-neighborhood approximations to these weights.

The local limit is a Poisson tree, which has branching structure and
independence properties we can exploit. The independence between
forward subtrees makes the Green's function recursion a closed
distributional system, known as a recursive distributional equation
(RDE). Such equations have been used to characterize limiting
Stieltjes transforms and prove convergence of empirical spectral
measures for locally convergent random graphs
\cite{bordenave2010resolvent}. Related Green's function methods have
also been developed to study absolutely continuous spectrum on trees
\cite{aizenman2006stability,keller2012absolutely,arras2023existence}.
Here we use individual Green's function entries to construct
approximate eigenvectors, with decay estimates allowing the weights
to be approximated from finite neighborhoods.

Another advantage of considering the infinite tree is that $H(t,T)$ is positive semidefinite for $t\in (-\infty,-1]\cup [1,\infty)$. Therefore, for $\delta>0$, its shifted inverse is well defined and bounded. This provides a canonical choice of  weights; the further decay and approximation estimates are what allow us to use them on the finite graph.

Given the weights $\Gamma_u^{(v)}$, we want to convert them into an eigenvector induced by the graph. To do this, we consider the planted labels as a Markov random field on the limiting tree. We go to some fixed depth $L$ and take a weighted sum of the planted values that we see there, with each weight given by a product of  weights along the path to that vertex. Our goal is to show that this estimation is stable. We therefore introduce an operator $\mathcal L$, which describes the change in the weighted construction when we pass from level $L$ to level $L+1$. For a suitable $\delta$, the relevant weighted average is preserved across levels. This stability allows the local approximations to satisfy the eigenvector equation in the limit.

In terms of the correct $\delta$, we have two conflicting parts that must balance out. The first is the rate of decay of covariance in the planted label field. Specifically, for a normalized planted label function associated with an eigenvalue $\mu$, the covariance on the broadcast tree is $(\mu/d)^{\operatorname{dist}(x,y)}$. To counteract this, we choose the Green's function weights so that the weighted signal remains at a stable scale. Heuristically, if the geometric average of edge weights is $g$, then the $d^\ell$ vertices in expectation at depth $\ell$ would contribute a signal of order
\[
d^\ell\left(\frac{\mu}{d}\right)^\ell g^\ell
=(\mu g)^\ell.
\]
Thus the matching weight would be $g=1/\mu$. If $\delta$ were smaller, the positive Green's function weights would be larger and the local formula for the eigenvector would not converge. If $\delta$ were larger, the Green's function would decay too quickly and the planted signal would disappear in the limit. At the correct $\delta$, the signal survives at a stable scale, allowing the approximations at neighboring vertices to describe the same eigenvector.

Previous analyses have used vectors derived from the Ihara-Bass formula and known structure of outlier eigenvectors of the non-backtracking operator, either directly \cite{Stephan2025community} or adapted to the outlier eigenvalue \cite{dong2026bethe}. In this work, we allow the weights to depend on both the local geometry and the spectral shift $\delta$. An appealing part of this approach is that it suggests a route to other models: construct the Green's function weights on the local limit, and the outlier eigenvalue corresponds to the shift that matches the planted signal.

\textbf{AI Acknowledgment:}
OpenAI's GPT-6 introduced the key analytic tools used in this argument. Specifically, the authors recently wrote a paper \cite{dong2026bethe} that recovered the number of outlier eigenvalues using a test function based on the local Markov random field. They gave this paper to GPT 6-Astra and asked whether these ideas could be used  to achieve eigenvector recovery. The response was a version of the general argument suggested here: to build a recursion based on the eigenvector equation and build test vectors based on this recursion. The authors realized that this was a Green's function recursion and, taking the GPT-6 argument as a preliminary draft, substantially restructured the paper with the relevant history and techniques given this interpretation. GPT-6 was further used to provide feedback on the manuscript and simplify explication.
\section{Main Results}
\subsection{The Stochastic Block Model}\label{subsec:sbm}

Given a symmetric nonnegative matrix $P \in \mathbb{R}^{r \times r}$ and a distribution $\pi$ over $[r]$ for $r\geq 2$ fixed, the stochastic block model $\mathrm{SBM}_n(P,\pi)$ generates a random graph $G$ on $n$ vertices as follows:
\begin{enumerate}
    \item Fix a label map $\sigma:[n]\to[r]$ such that $|\sigma^{-1}(i)|=\pi_i n$ for every $i\in[r]$.
    \item For every pair of vertices $u,v\in[n]$, include the edge $uv$ independently with probability $\frac{1}{n}P_{\sigma(u),\sigma(v)}$.
\end{enumerate}

The goal is to show when it is possible to detect the original communities. Specifically, we will focus on the \emph{sparse} regime, where the expected degree of each vertex is constant. 

Throughout, we assume that all vertices in $G$ have the same expected degree $d>1$, i.e., 
\begin{align}\label{eq:const-deg}
d\deq \sum_{j=1}^r P_{ij}\pi_j>1
\qquad
\text{is the same for all } i\in[r].
\end{align}
The restriction $d>1$ is necessary, as the weak-recovery threshold is at least $1$.

\subsection{Informative eigenvalues} \label{subsec:inform-ev}
Given $\mathrm{SBM}_n(P,\pi)$, define the \textit{signal matrix} $Q\deq P\Pi$, where $\Pi\deq \diag(\pi_1,\dots,\pi_r)$. Note that $Q_{ij}=P_{ij}\pi_j$ represents the expected number of neighbors in community $j$ of a vertex in community $i$. 

As $Q$ is similar to $\Pi^{1/2}P\Pi^{1/2}$, it has real eigenvalues. We order the eigenvalues of $Q$ by absolute value as
\begin{align}\label{eq:evs}
    d=\mu_1\geq|\mu_2|\geq\cdots\geq|\mu_r|.
\end{align}
\begin{dfn}\label{def:KS}
An eigenvalue $\mu$ of the signal matrix $Q$ is
\emph{informative} if $\mu^2>d$.
The model is above the Kesten--Stigum threshold if
there is an informative eigenfunction orthogonal
to the constant function in $L^2(\pi)$.
\end{dfn}

This is the natural threshold for linear reconstruction on the local Galton–Watson tree. For a nonzero eigenvalue $\mu_i$, the corresponding eigenfunction sums at depth $\ell$, normalized by $\mu_i^\ell$, form a martingale whose increment variances are proportional to $(d/\mu_i^2)^\ell$. These variances are summable precisely when $\mu_i^2>d$, so the normalized signal has bounded second moment exactly above this threshold.
Thus, $r_+$ and $r_-$ are the numbers of positive and negative informative eigenvalues, respectively.

\subsection{Our result}\label{subsec:main-result}

Let $G_n\sim\SBM_n(P,\pi)$, where the fixed parameters satisfy
\eqref{eq:const-deg}. Let $Q$ be the mean offspring matrix, and choose
real eigenfunctions $Q\varphi_i=\mu_i\varphi_i$ orthonormal in
$L^2(\pi)$, with $\varphi_1\propto\mathbf 1$.  
Write
\[
X_i(u)\deq \varphi_i(\sigma(u)).
\]
Thus $\langle X_i,X_j\rangle=n\mathbf 1_{\{i=j\}}$ under our
deterministic community-size convention. Positive and negative planted
eigenvalues are treated using $H(\sqrt d,G_n)$ and $H(-\sqrt d,G_n)$,
respectively.

The number of negative outliers is known throughout the fixed-degree
regime $d>1$ \cite{dong2026bethe}, while weak recovery from the
eigenvectors was previously established for sufficiently large $d$
\cite{Stephan2025community}. Our first result identifies the
corresponding eigenspaces throughout the fixed-degree regime.
We construct vectors from the planted
labels in local neighborhoods, with weights given by the Green's
function on the limiting tree, and show that these vectors approximate
the actual Bethe-Hessian eigenspaces. This also gives the limiting
outlier locations and the overlap with the planted directions.

\begin{thm}[Outlier eigenvalues and eigenspaces]\label{thm:eigenvectors}
Let $G_n\sim\SBM_n(P,\pi)$ have fixed parameters and common
expected degree $d>1$. List the informative planted eigenvalues,
counted with multiplicity, as
\[
\mu_1^+\ge\cdots\ge\mu_{r_+}^+>\sqrt d,
\qquad
\mu_1^-\le\cdots\le\mu_{r_-}^-<-\sqrt d.
\]
Order the eigenvalues of each Bethe Hessian increasingly.

There is a deterministic strictly increasing function
$\delta:(\sqrt d,d]\to(0,\infty)$ such that, as $n\to\infty$ and ordering eigenvalues from most negative to most positive,
\[
\begin{aligned}
\lambda_i(H(\sqrt d,G_n))
&\xrightarrow{\mathbb P}-\delta(\mu_i^+),
&& i\in[r_+],\\
\lambda_j(H(-\sqrt d,G_n))
&\xrightarrow{\mathbb P}-\delta(|\mu_j^-|),
&& j\in[r_-].
\end{aligned}
\]

For each distinct informative eigenvalue $\mu$, let $E_{\mu,n}$
be the span of the corresponding outlier eigenvectors. There exists
$c_\mu>0$, depending only on $Q$ and $\mu$, such that,
with probability tending to one,
\[
\|\operatorname{Proj}_{E_{\mu,n}}x\|_2
\ge c_\mu\|x\|_2
\qquad
\text{for every }x\in\operatorname{span}\{X_i:\mu_i=\mu\}.
\]
In particular, if $\mu_i$ is simple and $v_{i,n}$ is its
unit outlier eigenvector, then there exists a deterministic
$q_{\mu_i}>0$ such that
\begin{equation}\label{eq:eigenvector-overlap}
\frac{|\langle v_{i,n},X_i\rangle|}{\sqrt n}
=q_{\mu_i}+o_{\mathbb P}(1).
\end{equation}
\end{thm}
This says that each true eigenvector has nontrivial inner product with the planted eigenspace. Given this, we can use standard arguments to achieve weak recovery.

For labelings $\sigma,\widehat\sigma:[n]\to[r]$, define
the centered overlap
\[
\ov(\sigma,\widehat\sigma)
:=\max_{\tau\in S_r}\frac1n\sum_{u=1}^n
\left[
\mathbf1_{\{\sigma(u)=\tau(\widehat\sigma(u))\}}
-\pi_{\tau(\widehat\sigma(u))}
\right].
\]
For a fixed permutation $\tau$, the subtracted term is
the agreement expected from independent planted labels
with proportions $\pi$, keeping the predicted labels
fixed. We say that an algorithm achieves weak recovery
if its centered overlap is bounded below by a positive
constant with probability tending to one.
For balanced communities, this is the usual agreement
after relabeling, minus $1/r$.

\begin{thm}[Weak recovery]\label{thm:weak-recovery}
Assume that the SBM parameters are fixed and satisfy
\eqref{eq:const-deg}. Let $c_\mu>0$ be the constants in
\Cref{thm:eigenvectors}, and set 
$\pi_{\min}\deq \min_{a\in[r]}\pi_a.$ 
For every fixed choice of the SBM parameters, there is a
randomized polynomial-time algorithm, with constants
depending on those parameters, that rounds the informative
Bethe-Hessian embedding to an output $\widehat\sigma_n$
satisfying, with probability $1-o_n(1)$,
\[
\ov(\sigma,\widehat\sigma_n)
\geq \frac{\pi_{\min}}{64r^2(r-1)}\sum_{\substack{i\ge2\\|\mu_i|>\sqrt d}}c_{\mu_i}^2.
\]
In particular, weak recovery is possible whenever
\[
\max_{i\ge2}|\mu_i|>\sqrt d.
\]
\end{thm}
\begin{remark}
For the symmetric two-community model with
\[
\pi=(1/2,1/2),
\qquad
P=\begin{pmatrix}a&b\\b&a\end{pmatrix},
\]
the expected degree is $d=(a+b)/2$, and the nonconstant
planted eigenvalue is $\mu_2=(a-b)/2$. Thus
\Cref{thm:weak-recovery} gives weak recovery by randomized
rounding of Bethe-Hessian eigenvectors whenever
\[
(a-b)^2>2(a+b).
\]
This achieves the information-theoretically optimal
weak-recovery threshold, predicted in
\cite{decelle2011asymptotic} and established rigorously in
\cite{mossel2015reconstruction}.
\end{remark}

The theorem reaches the Kesten--Stigum threshold throughout the
fixed degree regime $d>1$.

The proof comes in three parts. First, we construct the Green's
function weights and study their transfer operator.
Proposition~\ref{prop:delta} identifies the shift $\delta(\mu)$
associated with each informative planted eigenvalue.
We then construct the corresponding fields on the limiting tree
and approximate them using finite neighborhoods.
Propositions~\ref{prop:finite-fields}
and~\ref{prop:eigenspace-approximation} transfer these fields
to the finite graph and identify the informative eigenspaces,
completing the proof of \Cref{thm:eigenvectors}.
Finally, \Cref{sec:recovery} proves \Cref{thm:weak-recovery}
by a randomized rounding of the spectral coordinates.
\section{Green's function weights}\label{sec:transfer}

In this section, we will consider the set of weights that work for the local limit of the SBM, the Poisson tree. Take  $G$ to be any graph which may be finite or infinite, and let $E:=\{(u,v):u\sim v\}$ denote the set of directed edges in $G$. Let $t\geq \sqrt{d}>1$,  $\delta>0$, and $c_\delta:=\frac{t}{t^2+\delta}$. 
 For $(T,o)$ a rooted tree, for any $u\in V(T)\setminus\{o\}$, let $p(u)$ denote the parent of $u$. For any $u\in V(T)$ let $\ch(u)$ denote the set of children of $u$.

We will build our resolvent recursively, starting from 0.

\begin{dfn}\label{def:gamma-finite}
    Given graph $G$, for every $M\in \bZ_{\geq 0}$ and $(u,v)\in E$, define the quantity $(\Gamma_u^{(v)})^{(M)}$ by
    \begin{align}\label{eq:gamma-recursion-finite}
        (\Gamma_u^{(v)})^{(0)}\deq0,\qquad (\Gamma_u^{(v)})^{(M+1)}\deq \frac{t}{t^2+\deg(u)-1+\delta-t\sum_{\substack{w\sim u\\w\neq v}} (\Gamma_w^{(u)})^{(M)}}.
    \end{align}
\end{dfn}

Our goal is to show that by passing to this iteration, \eqref{eq:gamma-recursion-finite} becomes a self-consistent equation. To do this, we note some properties of the recursion. 

\begin{lemma}\label{lem:gamma-increasing-bounded}
        For every $(u,v)\in E$, $0=(\Gamma^{(v)}_u)^{(0)}\leq (\Gamma^{(v)}_u)^{(1)}\leq\dots\leq c_\delta$. 
    \end{lemma}
    \begin{proof}
        We show by induction that $0=(\Gamma^{(v)}_u)^{(0)}\leq (\Gamma^{(v)}_u)^{(1)}\leq\dots$. Indeed, we have $(\Gamma^{(v)}_u)^{(0)}=0<(\Gamma^{(v)}_u)^{(1)}$ for all $(u,v)\in E$. Inductively, if $(\Gamma^{(v)}_u)^{(M-1)}\leq (\Gamma^{(v)}_u)^{(M)}$ for all $(u,v)\in E$, then $$\sum_{\substack{w\sim u\\w\neq v}} (\Gamma_w^{(u)})^{(M-1)}\leq \sum_{\substack{w\sim u\\w\neq v}} (\Gamma_w^{(u)})^{(M)},$$ which gives $(\Gamma^{(v)}_u)^{(M)}\leq (\Gamma^{(v)}_u)^{(M+1)}$.

        Furthermore, if $0\leq (\Gamma^{(v)}_u)^{(M)}\leq 1/t$ for all $(u,v)\in E$ (which is true for $M=0$), then we have
        \[
        t^2+\deg(u)-1+\delta-t\sum_{\substack{w\sim u\\w\neq v}} (\Gamma_w^{(u)})^{(M)}\geq t^2+\delta,
        \]
        which gives $ (\Gamma^{(v)}_u)^{(M+1)}\leq \frac{t}{t^2+\delta}=c_\delta<1/t$. Hence $0\leq (\Gamma^{(v)}_u)^{(M)}\leq c_\delta$ for all $M\geq 0$.
    \end{proof}

    The above lemma allows us to define the following.

    \begin{dfn}\label{def:gamma-limit}
        For every $(u,v)\in E$, define $\Gamma_u^{(v)}:=\lim_{M\to\infty}(\Gamma_u^{(v)})^{(M)}$. We know from Lemma \ref{lem:gamma-increasing-bounded} and \eqref{eq:gamma-recursion-finite} that
        \begin{align}\label{eq:gamma-recursion}
            \Gamma_u^{(v)}=\frac{t}{t^2+\deg(u)-1+\delta-t\sum_{\substack{w\sim u\\w\neq v}}\Gamma_w^{(u)}}   ,\qquad 0\leq \Gamma_u^{(v)}\leq c_\delta.
        \end{align}
        When $G$ is a rooted tree and $u$ is a non-root vertex, we will simply use $\Gamma_u$ to denote $\Gamma_u^{(p(u))}$ and use $\Gamma_u^{(M)}$ to denote $(\Gamma_u^{(p(u))})^{(M)}$. For the root $o$, we define $\Gamma_o^{(M)}$ analogously
without excluding any neighbor, and set
$\Gamma_o:=\lim_{M\to\infty}\Gamma_o^{(M)}$.
In particular,
\be\label{eq:root-gamma-recursion}
\Gamma_o=\frac{t}{t^2+\deg(o)+\delta
-t\sum_{w\sim o}\Gamma_w}.
\ee
    \end{dfn}

    \section{The transfer operator}

Given the definition of these Green's function random variables $\Gamma_u$, we now introduce an operator that gives the value $\phi(u)/\phi(v)$ in terms of a ratio at distance $L$ rather than $L-1$. This is the straightforward fact that on any graph, we can write the Green's function as a sum over neighbors (see, e.g. \cite[Lemma 8.3]{erdHos2017dynamical}). From now on, we restrict to $t\geq \sqrt{d}>1$.

\subsection{Rooted trees}

\begin{dfn}\label{def:Omega}
    Let $\Omega$ denote the collection of all locally finite rooted trees with root $o$. For every such tree $T\in \Omega$,  
    let $T_1,\dots,T_K$ be the $\deg(o)$ forward subtrees rooted at children of $o$. Also,
    define the map $\Gamma:\Omega\to (0,c_\delta]$ by $ \Gamma(T):=\Gamma_o$, with $\Gamma_o$ as in Definition \ref{def:gamma-limit}.
\end{dfn}
We associate with $\Omega$ the probability measure given by the Poisson Galton-Watson tree. 
\begin{dfn}\label{def:nu}
Let $\nu$ be the law on $\Omega$ of a Galton--Watson tree
with offspring distribution $\Poi(d)$. Write $L^2(\Omega)$ for
the space of measurable functions $f:\Omega\to\CC$ satisfying
\[
    \int_\Omega |f(T)|^2\,d\nu(T)<\infty,
\]
where functions that are equal $\nu$-almost surely are identified.
Equip this space with the inner product
\[
    \langle f,g\rangle
    :=\int_\Omega \overline{f(T)}\,g(T)\,d\nu(T)
    =\EE_{T\sim\nu}\bigl[\overline{f(T)}\,g(T)\bigr].
\]
\end{dfn}

We first show the uniqueness and continuity of the function $\Gamma$ on $\Omega$. 
\begin{lemma}\label{lem:gamma-uniqueness}
Fix $t\geq\sqrt d>1$. For each $\delta>0$, there is a
unique, up to $\nu$-almost-sure equality, measurable
function $\Gamma_\delta:\Omega\to[0,c_\delta]$
satisfying \eqref{eq:root-gamma-recursion}. 
Moreover, $\delta\mapsto\Gamma_\delta$ is continuous
in $L^1(\nu)$ on $(0,\infty)$.
\end{lemma}
\begin{proof}
    Suppose $\Gamma$ and $\Gamma'$ are two different functions that satisfy \eqref{eq:root-gamma-recursion}. Then we have
    \[
    \Gamma(T)=\frac{t}{t^2+K+\delta-t\sum_{j=1}^K\Gamma(T_j)},\qquad \Gamma'(T)=\frac{t}{t^2+K+\delta-t\sum_{j=1}^K\Gamma'(T_j)}.
    \]
    Write $\Gamma=\Gamma(T),\Gamma_j=\Gamma(T_j), K=K(T)$. Then 
    \[
    |\Gamma-\Gamma'|=\Gamma \Gamma'|\sum_{j=1}^K(\Gamma_j-\Gamma_j')|\leq c_\delta^2\sum_{j=1}^K|\Gamma_j-\Gamma_j'|.
    \]
    Since $T,T_j$ have the same law, we get that $\EE|\Gamma-\Gamma'|\leq dc_\delta^2 \EE|\Gamma-\Gamma'|$, so $\EE|\Gamma-\Gamma'|=0$. This shows uniqueness of $\Gamma$.

    Similarly, for the obtained $\Gamma$, we get that
    \[
    |\Gamma_{\delta}-\Gamma_{\delta_0}|=\Gamma_{\delta} \Gamma_{\delta_0}\left|\frac{\delta_0-\delta}{t}+\sum_{j=1}^K(\Gamma_\delta(T_j)-\Gamma_{\delta_0}(T_j))\right|\leq c_\delta c_{\delta_0}\left|\frac{\delta_0-\delta}{t}+\sum_{j=1}^K(\Gamma_\delta(T_j)-\Gamma_{\delta_0}(T_j))\right|.
    \]
    We obtain that
    \begin{align*}
        \EE |\Gamma_{\delta}-\Gamma_{\delta_0}|&\leq \frac{c_\delta c_{\delta_0}}{t}|\delta_0-\delta|+c_\delta c_{\delta_0} d\EE |\Gamma_{\delta}-\Gamma_{\delta_0}|,\\
        &\leq \frac{\frac{c_\delta c_{\delta_0}}{t}|\delta_0-\delta|}{1-dc_\delta c_{\delta_0}}\to 0 \text{ as $\delta\to\delta_0$.}
    \end{align*}
\end{proof}

We then have the following relation from perturbing $\delta$.

\begin{lemma}
    $\Gamma_{\delta}$ is strictly decreasing in $\delta$.   Moreover $\Gamma_{\delta}\uparrow 1/t$ as $\delta\downarrow 0$, almost surely.
\end{lemma}
\begin{proof}
    We first show that $\Gamma_{\delta}$ is nonincreasing. Let $\delta_1<\delta_2$. Recall that $\Gamma^{(0)}_{\delta_1}=\Gamma^{(0)}_{\delta_2}=0$. It is also easy to check from the recursion that if $\Gamma^{(m)}_{\delta_1}\geq \Gamma^{(m)}_{\delta_2}$, then $\Gamma^{(m+1)}_{\delta_1}\geq \Gamma^{(m+1)}_{\delta_2}$.

    Now observe that 
    \[
    \Gamma_{\delta_1}(T)- \Gamma_{\delta_2}(T)=\Gamma_{\delta_1}(T) \Gamma_{\delta_2}(T)\left(\frac{\delta_2-\delta_1}{t}+\sum_{j=1}^K( \Gamma_{\delta_1}(T_j)- \Gamma_{\delta_2}(T_j)) \right).
    \]
    Since $\Gamma_{\delta_1}(T) \Gamma_{\delta_2}(T)>0$, the above is $>0$. 

    Since all $\Gamma_{\delta}<1/t$, let $\Gamma_0\deq \lim_{\delta\downarrow 0}\Gamma_\delta$ (the pointwise limit function). Letting $e\deq 1-t\Gamma_0$, we know from the recursion formula that
    \begin{align*}
        e=1-\frac{t^2}{t^2+\sum_{j=1}^K e_j }=\frac{\sum_{j=1}^K e_j}{t^2+\sum_{j=1}^K e_j }.
    \end{align*}
    If $\EE e>0$, then we have $\EE e< \EE\left(t^{-2}\sum_{j=1}^K e_j\right) = t^{-2}d \EE e\leq \EE e$, a contradiction. Hence $\EE e=0$ and $\Gamma_0=1/t$ almost surely.
\end{proof}

All future stated (in)equalities in $L^2(\Omega)$ occur with probability 1.

\subsection{Introducing the transfer operator}

\begin{dfn}
    On  $L^2(\Omega)$, define the operator $\cL:L^2(\Omega)\to L^2(\Omega)$ by
    \[
    (\cL f)(T):=\Gamma(T)\sum_{j=1}^K f(T_j).
    \]
\end{dfn}

\begin{obs}
The above map is well-defined. In other words, $\|\cL\| <\infty$.
\end{obs}
\begin{proof}
    Recall from the previous section that $\Gamma(T)\leq c_\delta$ for all $T$. Thus 
    \begin{align}\label{eq:L-norm}
        \|\cL f\|_2^2=\EE_{T\sim\nu}[\Gamma(T)^2|\sum_{j=1}^{\deg(o)}f(T_j)|^2]\leq c_\delta^2(d\|f\|_2^2+d^2|\EE f|^2)\leq c_\delta^2(d^2+d)\|f\|_2^2.
    \end{align}
\end{proof}

Our next goal will be to show that this transfer operator has a spectral gap. Therefore iterating it sufficiently will cause convergence to the top eigenspace. 
\begin{dfn}
    Let $\kappa$ denote the spectral radius of $\cL$. That is, $\kappa=\sup\{|z|:z\in\sigma(\cL)\}$ where $\sigma$ is the spectrum of $\cL$.
    Of course, $\kappa\leq \|\cL\| <\infty$.
\end{dfn}
We now want to show in our regime of interest properties of the top eigenspace. Let $\Pi$ be the one-dimensional projection operator, and define the operator $\wh{\cL}$ by
\[
\Pi f\deq(\EE_\nu f)\mathbf 1,\qquad \wh{\cL}\deq {\cL}(I-\Pi).\] 
We first give an upper bound on $\|\wh{\cL}\|$. Consider any $f\in L^2(\Omega)$ and let $g=f-\EE f$. Since $\EE g=0$, we know from the calculation of \eqref{eq:L-norm} that
\begin{align*}
    \|\cL g\|_2^2\leq c_\delta^2\EE_{T\sim \nu}\left(|\sum_{j=1}^K g(T_j)|^2\right)=dc_\delta^2\|g\|_2^2.
\end{align*}
Hence $\|\wh\cL\|\leq \sqrt{d}c_\delta$. Note that $\sqrt{d}c_\delta=\frac{\sqrt{d}t}{t^2+\delta}<1$. We can use this to give a spectral gap on $\cL$.

\begin{prop}\label{prop:kappa}
Suppose that the spectral radius $\kappa$ of $\cL$ satisfies
$\kappa>\sqrt d\,c_\delta$. Then:
\begin{enumerate}
    \item There exist nonzero nonnegative right and left
    eigenfunctions $h,\ell\in L^2(\Omega)$ satisfying
    \[
    \cL h=\kappa h,\qquad \cL^*\ell=\kappa\ell.
    \]
    Almost surely over $T\sim\nu$, we have $\ell(T)>0$, while $h(T)=0$ on finite trees
    and $h(T)>0$ on infinite trees.

    \item $\kappa$ is an algebraically simple eigenvalue
    of $\cL$.

    \item $\kappa$ is the only point of $\sigma(\cL)$
    with absolute value $\kappa$.

    \item There exist $C<\infty$ and $0<\theta<1$ such that,
    for every integer $L\geq0$,
    \[
    \|\kappa^{-L}\cL^L\mathbf 1-h\|_2\leq C\theta^L.
    \]
\end{enumerate}
\end{prop}

\begin{proof}
\textbf{Proof of (1).}
 Since
\[
\cL f=\wh{\cL}f+(\EE f)\cL\mathbf 1,
\qquad \|\wh{\cL}\|\leq \sqrt d\,c_\delta<\kappa,
\]
$\cL$ is a rank-one perturbation of an operator of norm at
most $\sqrt d\,c_\delta$. Thus $r_{\mathrm{ess}}(\cL)\leq \sqrt d\,c_\delta$, and every
spectral point outside $\{|z|\leq \sqrt d\,c_\delta\}$ is an isolated
eigenvalue of finite algebraic multiplicity. The generalized
Krein--Rutman theorem
\cite[Appendix Theorem~2.4 and Remark on p.~314]
{schaefer1999topological}
therefore gives nonzero $h,\ell\geq0$ satisfying
\[
\cL h=\kappa h,\qquad \cL^*\ell=\kappa\ell.
\]

Let $K(T)$ be the degree of the root of $T$. Then
\[
\kappa h(T)=\Gamma(T)\sum_{j=1}^{K(T) }h(T_j)
\]
shows by induction on height that $h(T)=0$ for finite $T$.
Since $\Gamma>0$, the probability $p=\nu(h(T)=0)$ satisfies
$p=\EE[p^{K(T)}]=\re^{d(p-1)}$. As $h\centernot{\equiv} 0$, $p$ equals the
extinction probability of the Galton-Watson tree. Hence $h>0$ almost surely on
infinite trees.

We next show that $\ell$ cannot be supported exclusively on finite
trees. Let $P_{\fin}$ be multiplication by their indicator and
consider $\cL_{\fin}\deq P_{\fin}\cL P_{\fin}$. As above, $\cL_{\fin}$ is a
rank-one perturbation of an operator of norm at most $\sqrt d\,c_\delta$,
so every spectral point outside $\{|z|\leq \sqrt d\,c_\delta\}$ is an
eigenvalue. Induction on tree height shows that $\cL_{\fin}$ has
no nonzero eigenvalues. Consequently,
$\sigma(\cL_{\fin})\subseteq\{|z|\leq \sqrt d\,c_\delta\}$. If $P_{\fin}\ell=\ell$, however,
then
\[
\cL_{\fin}^*\ell=P_{\fin}\cL^*P_{\fin}\ell=\kappa\ell,
\]
contradicting $\kappa>\sqrt d\,c_\delta$.

Now let $A\deq \{T:\ell(T)=0\}$ and suppose $p=\nu(A)>0$.
Positivity and the fact that
\[
\langle\ell,\cL^m\mathbf 1_A\rangle
=\kappa^m\langle\ell,\mathbf 1_A\rangle=0
\]
imply that, almost surely, whenever $\ell(T)>0$, every
depth-$m$ subtree avoids $A$. Writing $Z_m$ for the size
of generation $m$, the branching property gives
\[
\nu(\ell>0,\ |T|=\infty)
\leq
\EE\big[(1-p)^{Z_m}\mathbf 1_{\{Z_m>0\}}\big]
\longrightarrow0.
\]
Indeed, $Z_m\to\infty$ on survival and is eventually zero
on extinction. This would force $\ell$ to be supported on finite trees,
contrary to the preceding paragraph. Thus $\ell(T)>0$
almost surely.

\textbf{Proof of (2).}
The $\kappa$-eigenspace is one-dimensional: otherwise it
contains a nonzero $g$ with $\EE g=0$, giving a contradiction, as
\[
\kappa\|g\|_2=\|\cL g\|_2
=\|\wh{\cL}g\|_2\leq \sqrt d\,c_\delta\|g\|_2.
\]
A nontrivial Jordan chain would then give
$(\cL-\kappa I)g=h$ after rescaling, whereas
\[
0<\langle\ell,h\rangle
=\langle\ell,(\cL-\kappa I)g\rangle=0.
\]
Hence $\kappa$ is algebraically simple.

\textbf{Proof of (3).}
Suppose $\lambda\in\sigma(\cL)$ and $|\lambda|=\kappa$.
Since $\kappa>\sqrt d\,c_\delta$, $\lambda$ is an eigenvalue; choose
$f\neq0$ with $\cL f=\lambda f$. Then
\[
(\cL-\kappa I)|f|\geq0,
\qquad \langle\ell,(\cL-\kappa I)|f|\rangle=0.
\]
Since $\ell>0$ almost surely, $(\cL-\kappa I)|f|=0$, and (2) means we can rescale $f$ so that 
$|f|=h$. Write $f=h\circ \zeta$, where $\circ$ is the entrywise product and 
$|\zeta|=1$ on infinite trees. Now using the fact that $f$ and $h$ are eigenvectors of $\cL$, we get that
\[
\Gamma(T)\left|\sum_{j=1}^{K(T)} h(T_j)\zeta(T_j)\right|=\kappa h(T)
=\Gamma(T)\sum_{j=1}^{K(T)} h(T_j).
\]
Whenever $K(T)=2$ where the two child trees are infinite,
their phases therefore agree almost surely. These phases
are independent with the same law, so that law is a
point mass. Thus $f$ is a multiple of $h$, and
$\lambda=\kappa$.

\textbf{Proof of (4).}
By (2), the spectral projection at $\kappa$ is 
\[P\deq (\ell^*h)^{-1}h\ell^*.
\]
The spectral decomposition in (1) and (3) implies  that the
remaining spectrum lies in a strictly smaller disk.
Thus, the spectral radius of $(\cL-\kappa P)$ is strictly less than $\kappa$. Also, 
$P(\cL-\kappa P)=(\cL-\kappa P)P=0$, and
\[
\kappa^{-m}\cL^m-P=\kappa^{-m}(\cL-\kappa P)^m
\qquad (m\geq1).
\]
The spectral radius formula therefore gives, for some
$C<\infty$ and $0<\theta<1$,
\[
\|\kappa^{-m}\cL^m-P\| \leq C\theta^m.
\]
We can normalize so that $P\mathbf 1=h$, proving the
asserted convergence.
\end{proof}

\subsection{Injectivity of the largest eigenvalue}

Recall that $\kappa\deq \kappa(\delta)$ is always associated with $\delta,t$. Throughout $t$ is fixed. In this section we let $\delta$ vary. We now have the following.

\begin{prop}\label{prop:delta}
$\kappa(\delta)$ is continuous at all $\delta_0$ with $\kappa(\delta_0)>\sqrt{d}c_{\delta_0}$. 
    For every $t<\mu\leq d$, there is a unique $\delta=\delta(\mu)>0$ such that
    $\kappa(\delta)=d/\mu$.
    The function $\delta(\mu)$ is strictly increasing.
\end{prop}
\begin{proof}
We extend $\Gamma_\delta$ to $\delta=0$ by setting
$\Gamma_0=1/t$.

\paragraph{Continuity.}
The rank-one decomposition
\[
\cL_\delta f
=\wh{\cL}_\delta f
 +(\EE f)\cL_\delta\mathbf 1
\]
gives, for $|z|>\|\wh{\cL}_\delta\|$,
\[
z\in\sigma(\cL_\delta)
\quad\Longleftrightarrow\quad
\EE\big[(z-\wh{\cL}_\delta)^{-1}
             \cL_\delta\mathbf 1\big]=1.
\]
Dominated convergence gives
\[
\cL_\delta\mathbf 1\longrightarrow
\cL_{\delta_0}\mathbf 1
\quad\text{in }L^2,
\qquad
\wh{\cL}_\delta\longrightarrow
\wh{\cL}_{\delta_0}
\quad\text{strongly}.
\]
If for some $b>0$, $\|\wh{\cL}_\delta\|\leq b$ locally uniformly in $\delta$,
the Neumann series $(z-\wh{\cL}_\delta)^{-1}=z^{-1}+\wh{\cL}_\delta z^{-2}+\cdots$ therefore yields
\[
\EE\big[(z-\wh{\cL}_\delta)^{-1}
             \cL_\delta\mathbf 1\big]
\longrightarrow
\EE\big[(z-\wh{\cL}_{\delta_0})^{-1}
             \cL_{\delta_0}\mathbf 1\big]
\]
locally uniformly on $\{|z|>b\}$.
Rouch\'e's theorem, applied to these expressions minus $1$,
shows that every spectral point of $\cL_{\delta_0}$ in this
region persists nearby. Conversely, every limit in this
region of spectral points of $\cL_\delta$ belongs to
$\sigma(\cL_{\delta_0})$.
Since the spectra are locally uniformly bounded, it follows that
\[
\max\{\kappa(\delta),b\}
\longrightarrow
\max\{\kappa(\delta_0),b\}.
\]
If $\kappa(\delta_0)>\sqrt d\,c_{\delta_0}$, continuity of
$c_\delta$ allows us to choose $b$ such that, locally,
\[
\sqrt d\,c_\delta<b<\kappa(\delta_0).
\]
This proves continuity of $\kappa$ at $\delta_0$.

\paragraph{Endpoint limits.}
At $\delta=0$,
\[
(\cL_0f)(T)=\frac1t\sum_{j=1}^K f(T_j),
\qquad
\EE[\cL_0f]=\frac dt\,\EE f.
\]
Thus $\EE[\wh{\cL}_0f]=0$, and the Neumann series gives
\[
\EE\big[(z-\wh{\cL}_0)^{-1}\cL_0\mathbf 1\big]
=\frac{d}{tz},
\qquad |z|>\frac{\sqrt d}{t}.
\]
The spectral criterion above therefore shows that $d/t$ is
the only spectral point outside this disk. Hence
$\kappa(0)=d/t$, and continuity at zero gives
\[
\lim_{\delta\downarrow0}\kappa(\delta)=\frac dt.
\]
At the other endpoint,
\[
\kappa(\delta)\leq\|\cL_\delta\|
\leq c_\delta\sqrt{d+d^2}\longrightarrow0
\qquad\text{as }\delta\to\infty.
\]

\paragraph{Existence and uniqueness.}
Positivity and the pointwise decrease of $\Gamma_\delta$
give $\cL_{\delta_2}\leq\cL_{\delta_1}$ whenever
$\delta_1<\delta_2$, so $\kappa$ is nonincreasing.

Fix $y\in[1,d/t)$. By the limit at zero, choose $\eta>0$
such that $\kappa(\eta)>y$. For every $\delta\ge\eta$,
\[
\|\wh{\cL}_\delta\|\le b:=\sqrt d\,c_\eta<1.
\]
Thus $\max\{\kappa(\delta),b\}$ is continuous on
$[\eta,\infty)$, starts above $y$, and converges to $b<y$.
It follows that $\kappa$ attains every value in $[1,d/t)$
at a positive parameter.

To prove uniqueness, suppose $\delta_1<\delta_2$ and
$\kappa(\delta_1),\kappa(\delta_2)\geq1$.
Both parameters satisfy the conditions of \Cref{prop:kappa}, so we take the corresponding left eigenfunction
$\ell_{\delta_1}$ and right eigenfunction $h_{\delta_2}$.
Consequently,
\[
\begin{aligned}
&\big(\kappa(\delta_1)-\kappa(\delta_2)\big)
 \langle\ell_{\delta_1},h_{\delta_2}\rangle\\
&\qquad=
\langle\ell_{\delta_1},
(\cL_{\delta_1}-\cL_{\delta_2})h_{\delta_2}\rangle>0,
\end{aligned}
\]
because $\Gamma_{\delta_1}>\Gamma_{\delta_2}$ and
$h_{\delta_2}>0$ on infinite trees.
Thus $\kappa$ is strictly decreasing wherever $\kappa\geq1$.

For $t<\mu\leq d$, we have $d/\mu\in[1,d/t)$, so there is
a unique $\delta(\mu)>0$ satisfying
$\kappa(\delta(\mu))=d/\mu$.
Since $d/\mu$ strictly decreases with $\mu$,
$\delta(\mu)$ strictly increases.
\end{proof}

\section{Boundary messages}

Given the weights and the spectral properties of the transfer operator, we now want to consider the messages given by the tree with these weights. In the following two sections, let $\mu_i> t$ be any eigenvalue of $Q$ with eigenvector $\varphi_i$.  Let $\rho=\rho_i\deq \mu_i/d$, $\delta=\delta_i:=\delta(\mu_i)>0$.

Let $G$ be a graph with vertex labeling $\sigma:V(G)\to[r]$. Define   $X=X_i:V(G)\to \RR$ by $(X_i)_v=\varphi_i(\sigma(v))$. Observe that, if $\sigma(v)\sim \pi$, then we have
\begin{align}\label{eq:Xi-orthonormal}
\EE[(X_i)_v(X_j)_v]=\sum_{a=1}^r\pi_a\varphi_i(a)\varphi_j(a)=\mathbf 1_{\{i=j\}}.
\end{align}

We introduce the  labelled  analogues of $\Omega$ and $\nu$ defined in the previous section.

\begin{dfn}
    Let $\widetilde \Omega$ denote the collection of all vertex-labeled rooted trees with root $o$, with labels in $[r]$.
    Let $\widetilde\nu$ be the law on $\widetilde\Omega$ of a multi-type Galton--Watson tree, where the root $o$ has type $\sigma(o)\sim\pi$, and each vertex of type $i$ independently has $\Poi(P_{ij}\pi_j)$ children of type $j$ for every $j\in[r]$.
\end{dfn}

Observe that, for any $T\sim\widetilde \nu$, the underlying graph $T_0$ is distributed as $\nu$.

We now define the boundary weights associated with the eigenvector. To do this, we will do the usual Green's function decomposition into a product of cavity roots (e.g.  \cite[Proposition 3.1]{aizenman2013resonant}). We also use $(\Gamma_u^{(v)})^{(M)}$ and $\Gamma_u^{(v)}$ defined in Definitions \ref{def:gamma-finite} and \ref{def:gamma-limit}.

\begin{dfn}\label{def:b}
    For every $(u,v)\in E$ and $M,L\geq 0$, define $b_{u\to v}^{(L,M)}$ and $b_{u\to v}^{(L)}$ by
\begin{align*}\label{eq:bL}
    &b_{u\to v}^{(0,M)}\deq X_u,\qquad b_{u\to v}^{(L+1,M)}\deq (\Gamma_u^{(v)})^{(M)}\sum_{\substack{w\sim u\\w\neq v}}b_{w\to u}^{(L,M)};\\
    &b_{u\to v}^{(0)}\deq X_u,\qquad b_{u\to v}^{(L+1)}\deq \Gamma_u^{(v)} \sum_{\substack{w\sim u\\w\neq v}}b_{w\to u}^{(L)}.
\end{align*}
When $G$ is a rooted tree, we use $b_{u}^{(L,M)}$ to denote $b_{u\to p(u)}^{(L,M)}$ and use $b_{u}^{(L)}$ to denote $b_{u\to p(u)}^{(L)}$. 
\end{dfn}

\begin{prop}\label{prop:bL-converge}
Suppose $G\sim\widetilde\nu$ is the multitype Galton--Watson tree. There is an $L^2$ function $b_{u}:\widetilde \Omega\to \RR$ such that 
    \be\label{eq:blimitrecursion}
\lim_{L\to\infty }b_{u}^{(L)}=b_{u},\qquad b_{u}=\Gamma_u\sum_{w\in\ch(u)}b_{w}.
    \ee
\end{prop}

\begin{proof}  Fix $\rho\deq \rho_i= \mu_i/d$. Set $h_L=\rho^L\cL^L\mathbf 1$.
Suppose $x$ is a descendant of $u$ with $|x|:=\dist(x,u)=L$. Let $u=x_0,x_1,\dots,x_L=x$ denote the path from $u$ to $x$. We then have
\begin{align*}
 h_L(T)=\rho^L\sum_{x:|x|=L}\prod_{j=0}^{L-1}\Gamma_{x_j}.
\end{align*}
    Write $P_u(x):=\prod_{j=0}^{L-1}\Gamma_{x_j}$. Then we have
    \[
    b_{u}^{(L)}=\sum_{|x|=L}P_u(x)X_x.
    \]
    Let $T_0$ denote the underlying unlabelled tree of $T$.
    Recall that  $\EE[X_x\mid T_0,\sigma(u)]=\rho^LX_u$. Thus, $$\EE[b_{u}^{(L)}\mid T_0,\sigma(u)]=\rho^L\sum_{|x|=L}P_u(x)X_u=h_L(T)X_u.$$ 
    
We now want to decompose $b_u^{(L)}$ into an orthogonal family, as was done in \cite{dong2026bethe}. 
For every descendant $x$ of $u$, define $\xi_x:=\rho^{-1}X_x-X_{p(x)}$, so that $X_x=\rho X_{p(x)}+\rho \xi_x$. Since
    $\EE(\xi_x\mid T_0,\text{ labels above $x$})=0$,
   \begin{align}\label{eq:xi-orthogonal}
       \EE(X_u\xi_x\mid T_0)=0,\qquad \EE(\xi_x\xi_y\mid T_0)=0\text{ for $x\neq y$.}
   \end{align}
    Meanwhile, we have $\EE(X_u^2\mid T_0)=1$ and
    $\EE(\xi_x^2\mid T_0)=\rho^{-2}\cdot 1-2\rho^{-1}\cdot \rho+1=\rho^{-2}-1$.
    Next, observe that for every path $u=x_0,x_1,\dots,x_L=x$, we have 
    \[
    X_x=\rho^L X_u+\sum_{\ell=1}^L \rho^{L-\ell+1}\xi_{x_\ell}.
    \]
  We can then decompose $b_u^{(L)}$ into this basis as
    \begin{align*}
        b_{u}^{(L)}&=\sum_{|x|=L}P_u(x)\left(\rho^L X_u+\sum_{\ell=1}^L \rho^{L-\ell+1}\xi_{x_\ell}\right)\\
        &=\rho^L\sum_{|x|=L}P_u(x)X_u+\sum_{\ell=1}^L\sum_{|y|=\ell} \rho^{L-\ell+1}\xi_y P_u(y)\sum_{|x|=L,y\preceq x}P_y(x)\\
        &=h_L(T)X_u+\sum_{\ell=1}^L\sum_{|y|=\ell}\rho h_{L-\ell}(T_y)P_u(y)\xi_y,
    \end{align*}
where $T_y$ is the subtree of $T$ rooted at $y$, and $y\preceq x$ denotes that $y$ is an ancestor of $x$.
We now define
\begin{align*}
     b_{u}(T)&\deq h(T)X_u+\sum_{\ell=1}^\infty\sum_{|y|=\ell}\rho h(T_y)P_u(y)\xi_y.
\end{align*}
We first show that $b_{u}\in L^2(\widetilde \Omega)$. For every $\ell$, let
\[
Y_\ell:=\sum_{|y|=\ell}\rho h(T_y)P_u(y)\xi_y.
\]
We know from orthogonality \eqref{eq:xi-orthogonal} that
\begin{align*}
    \EE(Y_\ell^2\mid T_0)&=\rho^2\sum_{|y|=\ell} h(T_y)^2P_u(y)^2\EE(\xi_y^2\mid T_0)=(1-\rho^2)\sum_{|y|=\ell} h(T_y)^2P_u(y)^2.
\end{align*}
Therefore, 
\begin{align*}
 \EE(Y_\ell^2)\leq (1-\rho^2)c_\delta^{2\ell}\EE\sum_{|y|=\ell} h(T_y)^2 \leq  (1-\rho^2)c_\delta^{2\ell}d^\ell\|h\|_2^2.
\end{align*}
Since the $Y_\ell$'s are pairwise orthogonal in expectation, it follows that
\[
\EE |b_{u}|^2\leq \|h\|_2^2\left(1+ (1-\rho^2)\sum_{\ell=1}^\infty c_\delta^{2\ell}d^\ell \right)<\infty.
\]

Next, we show that $b^{(L)}\to b$ in $L^2.$ Indeed, we have
\begin{align*}
    b^{(L)}(T)-b(T)&=(h_L(T)-h(T))X_u+\sum_{\ell=1}^L\sum_{|y|=\ell}\rho (h_{L-\ell}(T_y)-h(T_y)) P_u(y)\xi_y\\
    &\qquad -\sum_{\ell=L+1}^\infty \sum_{|y|=\ell}\rho  h(T_y) P_u(y)\xi_y.
\end{align*}
Using orthogonality of $\xi_y$, we get that
\begin{align*}
    \EE|b^{(L)}-b|^2&\leq \|h_{L}-h\|_2^2+(1-\rho^2) \sum_{\ell=1}^L c_\delta^{2\ell} d^\ell\|h_{L-\ell}-h\|_2^2+(1-\rho^2) \sum_{\ell=L+1}^\infty c_\delta^{2\ell} d^\ell\|h\|_2^2\to 0
\end{align*}
as $L\to\infty$.

We are left to show that $b$ satisfies the desired recursion. Let $K=|\ch(u)|$. Indeed, since
\begin{align*}
    &b_u^{(L+1)}=\Gamma_u\sum_{w\in\ch(u)}b_w^{(L)},\\
    &\EE|b_u-b_u^{(L+1)}|^2\to 0,\\
    &\EE|\Gamma_u\sum_{w\in\ch(u)}(b_w^{(L)}-b_w)|^2\leq c_\delta^2\EE(K^2)\EE|b_w^{(L)}-b_w|^2\to 0,
\end{align*}
we know that $b_u =\Gamma_u\sum_{w\in\ch(u)}b_w $ almost surely.
\end{proof}

We extend the definitions of $b_u=b_{u\to p(u)}$ to $b_{u\to v}$ for all $u\sim v$ in $T$.

\begin{dfn}\label{def:b-tree}
    For $u\sim v$ in $T$, define $b_{u\to v}$ by
    \[
    b_{u\to v}=\begin{cases}
        b_u &v=p(u)\\
\Gamma_o^{(v)}
\sum_{\substack{w\in\operatorname{ch}(o)\\w\neq v}}
b_{w\to o} & u=o, o=p(v)\\
\Gamma_u^{(v)}
\left(
b_{p(u)\to u}
+
\sum_{\substack{w\in\operatorname{ch}(u)\\w\neq v}}
b_{w\to u}
\right) & u\neq o, u=p(v)
    \end{cases}.
    \]
\end{dfn}

It is straightforward to check that, for every directed edge
$(u,v)$, we have
\[
b_{u\to v}
=
\Gamma_u^{(v)}
\sum_{\substack{w\sim u\\w\neq v}}b_{w\to u}.
\]
Moreover, for all sufficiently large $L$, each $b_{u\to v}$ can be expressed as a finite linear combination
of the messages $(b_{w\to p(w)})_{w\in T\setminus\{o\}}$ in the same way that $b_{u\to v}^{(L)}$ can be expressed in terms of
$(b^{(L-k)}_{w\to p(w)})_{w\in T\setminus\{o\}}$, with  $k$ independent of $L$. 
Thus, we can deduce the following from
Proposition \ref{prop:bL-converge}.

\begin{cor}\label{cor:bL-converge-all}
    Suppose $G$ is a rooted multitype Galton--Watson tree.
There exist random variables $(b_{u\to v})_{(u,v)\in E}$
such that, for every $(u,v)$,
\[
b_{u\to v}
=
\Gamma_u^{(v)}
\sum_{\substack{w\sim u\\w\neq v}}b_{w\to u}.
\]
Moreover, 
$b_{u\to v}^{(L)}
\longrightarrow b_{u\to v}$ in $L^2$
as $L\to\infty$.
\end{cor}

\section{Vertex fields and their approximations}
Given our set of messages $b_{u\rightarrow v}$, we can construct the eigenvector on the relevant graphs. We do this first on an arbitrary graph, and  then highlight stronger recursive  properties on trees. 
\subsection{General graphs}

Our test eigenvector is a weighted linear combination of the above messages.
\begin{dfn}
    Let $G$ be any graph with vertex labeling $\sigma:V(G)\to[r]$. For any $u\in V(G)$ and $L,M\geq 0$, define
    \begin{align*}
        V^{(L,M)}(u)
&\deq
\frac{
t\sum_{v\sim u}b_{v\to u}^{(L,M)}
}{
t^2+\deg(u)-1+\delta
-t\sum_{v\sim u}(\Gamma_v^{(u)})^{(M)}
},\\
V^{(L)}(u)
&\deq 
\frac{
t\sum_{v\sim u}b_{v\to u}^{(L)}
}{
t^2+\deg(u)-1+\delta
-t\sum_{v\sim u}\Gamma_v^{(u)}
}.
    \end{align*}

    Moreover, define the ball $B_{k}(G,u)$ to be the subgraph of $G$ induced by the set of vertices of distance at most $k$ from $u\in V(G)$.
\end{dfn}

It is important to mention here that $V^{(L,M)}$ can be decided by a finite neighborhood. 
\begin{lem}\label{lem:local}
Let $G$ be any graph with vertex labeling $\sigma:V(G)\to[r]$.
    For fixed integers $L,M\geq 0$, the values $V^{(L,M)}(u)$ and
$\bigl[(H(t,G)+\delta I)V^{(L,M)}\bigr](u)$
are determined by the ball $B_{M+L+2}(G,u)$.
\end{lem}
\begin{proof}
    It is easy to check by induction that every  $(\Gamma_u^{(v)})^{(M)}$ is determined by
$B_M(G,u)$. Since
\[
b_{u\to v}^{(0,M)}=X(u),
\qquad
b_{u\to v}^{(L+1,M)}
=
(\Gamma_u^{(v)})^{(M)}
\sum_{\substack{w\sim u\\w\neq v}}b_{w\to u}^{(L,M)},
\]
induction on $L$ then shows that $b_{u\to v}^{(L,M)}$
is determined by $B_{L+M}(G,u)$.
This gives that $V^{(L,M)}(u)$ is determined by $B_{L+M+1}(G,u)$.
Therefore, 
\[
\bigl[(H(t,G)+\delta I)V^{(L,M)}\bigr](u)
=
\bigl(t^2+\deg(u)-1+\delta\bigr)V^{(L,M)}(u)
-
t\sum_{v\sim u}V^{(L,M)}(v)
\]
is determined by $B_{L+M+2}(G,u)$.
\end{proof}

At this point, we write $c\deq c_\delta$ for notational convenience. We now give a generic rough bound on these weights.
\begin{lem}Let $G$ be any graph with vertex labeling $\sigma:V(G)\to[r]$.
    For every $L,M\geq 0$, $|b_{u\to v}^{(L,M)}|\leq \|\varphi_i\|_\infty c^{L}|B_{L}(G,u)|^{L}$.
\end{lem}
\begin{proof}Write $\beta:=\|\varphi_i\|_\infty$.  When $L=0$, we indeed have $|b^{(0,M)}_{u\to v}|\leq \beta$. Inductively, we obtain
\begin{align*}
|b_{u\to v}^{(L+1,M)}|
&\leq
c\sum_{\substack{w\sim u\\w\neq v}}
\beta c^L |B_L(G,w)|^L\leq
\beta c^{L+1}\deg(u)
|B_{L+1}(G,u)|^L\leq
\beta c^{L+1}|B_{L+1}(G,u)|^{L+1}.
\end{align*}
\end{proof}

\begin{lem}\label{lem:polynomial-bounded}
Let $G$ be any graph with vertex labeling $\sigma:V(G)\to[r]$.
    For each
fixed $L$ there is $C_L<\infty$, independent of $M$, such that
\[
|V^{(L,M)}(u)|
+
\left|
\bigl[(H(t,G)+\delta I)V^{(L,M)}\bigr](u)
\right|
\leq C_L|B_{L+2}(G,u)|^{L+2}.
\]
\end{lem}
\begin{proof}Let $\alpha:=t^2-1+\delta$ and recall that $\beta=\|\varphi_i\|_\infty$.
    By the previous observation, we have
\begin{align*}
|V^{(L,M)}(u)|
&\leq \frac{t\sum_{v\sim u}|b_{v\to u}^{(L,M)}|}
{\alpha}\leq 
\frac{t \beta}{\alpha}
\sum_{v\sim u}  c^L|B_L(G,v)|^L\\
&\leq
\frac{t\beta c^L}{\alpha}\deg(u)
|B_{L+1}(G,u)|^L \leq
\frac{t\beta c^L}{\alpha}|B_{L+1}(G,u)|^{L+1}.
\end{align*}
Using
\[
\bigl[(H(t,G)+\delta I)V^{(L,M)}\bigr](u)
=
(\alpha+\deg(u))V^{(L,M)}(u)
-t\sum_{v\sim u}V^{(L,M)}(v),
\]
we conclude that
\begin{align*}
|V^{(L,M)}(u)|
+
\left|
[(H(t,G)+\delta I)V^{(L,M)}](u)
\right|& \leq
(1+\alpha+\deg(u))|V^{(L,M)}(u)|
+t\sum_{v\sim u}|V^{(L,M)}(v)|\\
& \leq
\frac{t\beta c^L}{\alpha} (1+\alpha+(1+t)\deg(u) )
|B_{L+2}(G,u)|^{L+1}\\
& \leq
\frac{t\beta c^L}{\alpha}(\alpha+2+t)|B_{L+2}(G,u)|^{L+2}.
\end{align*}
\end{proof}

\subsection{Trees}
We now examine this candidate eigenvector on the infinite tree. 
\begin{dfn}\label{def:V}
    On the Galton--Watson tree $T$, recall $b_{v\to u}$ defined in
Definition \ref{def:b-tree}. For every vertex $u$, define the vector $V$ such that
\[
V(u)
\deq 
\frac{
t \sum_{v\sim u}b_{v\to u}
}{
t^2+\deg(u)-1+\delta
-
t \sum_{v\sim u}\Gamma_v^{(u)}
}.
\]
\end{dfn}

\begin{lemma}\label{lem:vertex-field-convergence}
On the labeled Galton--Watson tree $(T,o)$, we have
$\EE|V(o)|^2<\infty$ and
\[
\EE\left|V^{(L)}(o)-V(o)\right|^2
\longrightarrow 0
\qquad\text{as }L\to\infty.
\]
\end{lemma}

\begin{proof}
We know from
the definitions of $V^{(L)}(o)$ and $V(o)$ that
\[
|V^{(L)}(o)-V(o)|
\leq
\frac{t}{\alpha}
\sum_{w\sim o}
|b_{w\to o}^{(L)}-b_{w\to o}|,\qquad |V(o)|
\leq
\frac{t}{\alpha}
\sum_{w\sim o}
|b_{w\to o}|.
\]
By Cauchy--Schwarz,
\[
\left|V^{(L)}(o)-V(o)\right|^2
\leq
\frac{t^2}{\alpha^2}
\deg(o)\sum_{w\sim o}
\left|b_{w\to o}^{(L)}-b_{w\to o}\right|^2, \qquad |V(o)|^2
\leq
\frac{t^2}{\alpha^2}
\deg(o)\sum_{w\sim o}
|b_{w\to o}|^2.
\]
Observe that conditional on $K=\deg(o)$, each $(b_{w\to o}^{(L)},b_{w\to o})$ has the same marginal distribution as $(b^{(L)},b)$, independent of the value of $K$. 
By Proposition \ref{prop:bL-converge}, we have $\EE \left|b ^{(L)}-b\right|^2\to 0$ and $\EE |b|^2<\infty$.  Since $K\sim\text{Poi}(d)$,  we get that $$\EE\left|V^{(L)}(o)-V(o)\right|^2\leq \frac{t^2}{\alpha^2}\EE(K^2)\EE \left|b ^{(L)}-b\right|^2 \to 0, \qquad \EE|V(o)|^2\leq \frac{t^2}{\alpha^2}
\EE(K^2)\EE \left|b\right|^2<\infty.$$
\end{proof}

We show that the vector $V$ satisfies the eigenvector equation on the infinite tree. 
\begin{prop}\label{prop:eigvecexact}
  We have
   \[(H(t,T)V)(u)=-\delta V(u).\]
\end{prop}
\begin{proof}
 We first assume that for each $u\sim v$,
    \be\label{eq:Vtob}
b_{u\rightarrow v}=    V(u)-\Gamma_{u}^{(v)}V(v).
    \ee
If this is indeed the case, then we can write Definition \ref{def:V} as
\[
V(u)=\frac{t\sum_{w\sim u} (V(w)-\Gamma_w^{(u)}V(u))}{t^2-1+\delta+\deg(u)-t\sum_{w\sim u}\Gamma_w^{(u)}}.
\]
Cross-multiplying gives the relation
\[
(t^2-1+\delta+\deg(u))V(u)=t\sum_{w\sim u}V(w),
\]
giving $H(t,T)V=-\delta V$.

We are left to show \eqref{eq:Vtob}.  To do this,  by \eqref{eq:gamma-recursion}, we can rewrite Definition \ref{def:V} as
    \be\label{eq:Vufirst}
    V(u)=\frac{t\sum_{w\sim u}b_{w\to u}}{t((\Gamma_{u}^{(v)})^{-1}-\Gamma_v^{(u)})}=\frac{\sum_{w\sim u}b_{w\to u}}{(\Gamma_{u}^{(v)})^{-1}-\Gamma_v^{(u)}}.
    \ee
    We know from Corollary \ref{cor:bL-converge-all} that $b_{u\to v}=\Gamma_u^{(v)}\sum_{\substack{w\sim u\\w\neq v}}b_{w\to u}$, which gives  
    \begin{align}
    \begin{split}
    \label{eq:Vu-decomp}
        V(u)&=\frac{(\Gamma_{u}^{(v)})^{-1}b_{u\rightarrow v}+b_{v\rightarrow u}}{(\Gamma_{u}^{(v)})^{-1}-\Gamma_v^{(u)}}=\frac{b_{u\rightarrow v}+\Gamma_{u}^{(v)}b_{v\rightarrow u}}{1-\Gamma_{u}^{(v)}\Gamma_v^{(u)}}.
    \end{split}
    \end{align}
      We then have
    \begin{align*}
    V(u)-\Gamma_u^{(v)}V(v)&=\frac{b_{u\rightarrow v}+\Gamma_{u}^{(v)}b_{v\rightarrow u}}{1-\Gamma_{u}^{(v)}\Gamma_v^{(u)}}-\Gamma_u^{(v)}\cdot \frac{b_{v\rightarrow u}+\Gamma_{v}^{(u)}b_{u\rightarrow v}}{1-\Gamma_{u}^{(v)}\Gamma_v^{(u)}}
    =b_{u\rightarrow v},
    \end{align*}
    giving \eqref{eq:Vtob}.
\end{proof}

We now show that this equality is satisfied by the limit of $V^{(L)}$ as well.

\begin{lemma}\label{lemma:residual}

We have $\EE|[(H(t,T)+\delta I) V^{(L)}](o)|^2\to 0$ as $L\to \infty$.
That is, we have
\begin{equation}\label{eq:residual-limit}
 \E\left|(t^2-1+\delta+\deg(o))V^{(L)}(o)-t\sum_{w\sim o}V^{(L)}(w)\right|^2
 \longrightarrow0 \qquad\text{as }L\to\infty.
\end{equation}
\end{lemma}

\begin{proof}
Set $e_{u\to v}^{(L)}\deq b_{u\to v}^{(L)}-b_{u\to v}^{(L+1)}$. We can expand Definition \ref{def:b}, then use the fact that \eqref{eq:Vufirst} is also true between $V^{(L)}$ and $b^{(L)}$, giving
\begin{align*}
    b^{(L)}_{u\to v}&= \Gamma_u^{(v)}\sum_{w\sim u, w\neq v}b_{w\rightarrow u}^{(L)}+e^{(L)}_{u\to v}=(1-\Gamma_u^{(v)}\Gamma_v^{(u)})V^{(L)}(u)-\Gamma_u^{(v)}b_{v\to u}^{(L)}+e^{(L)}_{u\to v}.
\end{align*}
Interchanging $u$ and $v$ gives $b_{v\to u}^{(L)}
=
\bigl(1-\Gamma_u^{(v)}\Gamma_v^{(u)}\bigr)V^{(L)}(v)
-\Gamma_v^{(u)}b_{u\to v}^{(L)}
+e_{v\to u}^{(L)}$.
Substituting this 
into the preceding equation, we obtain
\begin{align*}
b_{u\to v}^{(L)}
&=
\bigl(1-\Gamma_u^{(v)}\Gamma_v^{(u)}\bigr)V^{(L)}(u)-\Gamma_u^{(v)}
\left(
\bigl(1-\Gamma_u^{(v)}\Gamma_v^{(u)}\bigr)V^{(L)}(v)
-\Gamma_v^{(u)}b_{u\to v}^{(L)}
+e_{v\to u}^{(L)}
\right)
+e_{u\to v}^{(L)}\\
&=
\bigl(1-\Gamma_u^{(v)}\Gamma_v^{(u)}\bigr)
\left(V^{(L)}(u)-\Gamma_u^{(v)}V^{(L)}(v)\right)+\Gamma_u^{(v)}\Gamma_v^{(u)}b_{u\to v}^{(L)}
+e_{u\to v}^{(L)}
-\Gamma_u^{(v)}e_{v\to u}^{(L)},\\
&=
V^{(L)}(u)-\Gamma_u^{(v)}V^{(L)}(v)
+
\frac{
e_{u\to v}^{(L)}-\Gamma_u^{(v)}e_{v\to u}^{(L)}
}{
1-\Gamma_u^{(v)}\Gamma_v^{(u)}
}.
\end{align*}
Substituting this into Definition \ref{def:V} therefore yields
\begin{align*}
&\phantom{{}={}}\left(t^2-1+\delta+\deg(o)-t\sum_{w\sim o}\Gamma_w^{(o)}\right)V^{(L)}(o)\\
&=t\sum_{w\sim o}b_{w\to o}^{(L)}\\
&=
t\sum_{w\sim o}V^{(L)}(w)
-t\left(\sum_{w\sim o}\Gamma_w^{(o)}\right)V^{(L)}(o)+
t\sum_{w\sim o}
\frac{
e_{w\to o}^{(L)}-\Gamma_w^{(o)}e_{o\to w}^{(L)}
}{
1-\Gamma_w^{(o)}\Gamma_o^{(w)}
}. 
\end{align*}
Since $|\Gamma_u^{(v)}|\leq c<1$, Cauchy--Schwarz gives
\be\label{eq:ebasederror}
\E\left|
(t^2-1+\delta+\deg(o))V^{(L)}(o)-t\sum_{w\sim o}V^{(L)}(w)
\right|^2
\leq
C\,\E\left[
\deg(o)\sum_{w\sim o}
\left(
|e^{(L)}_{w\to o}|^2+
|e^{(L)}_{o\to w}|^2
\right)
\right]
\ee
where  $C$ depends only on $t,c$.
Since the incoming defects satisfy
$\EE|e_{w\to o}^{(L)}|^2=\EE|e^{(L)}|^2\to 0$
as $L\to\infty$, the outgoing defects satisfy
\[
e_{o\to w}^{(L)}
=
\Gamma_o^{(w)}
\sum_{\substack{w'\sim o\\w'\neq w}}
e_{w'\to o}^{(L-1)},
\]
$|\Gamma_o^{(w)}|\leq c$, and   $\deg(o)\sim\text{Poi}(d)$,
we know that \eqref{eq:ebasederror} tends to zero as $L\to\infty$.
\end{proof}

We now obtain similar approximations for the variable $V^{(L,M)}$, which only depends on a finite neighborhood.
\begin{lemma}\label{lem:finite-field-approximation}
On the Galton--Watson tree, for every $L\geq 0$, as $M\to\infty$, we have
\[
V^{(L,M)}(o)\longrightarrow V^{(L)}(o)
\qquad \text{and}\qquad
\bigl[(H+\delta I)V^{(L,M)}\bigr](o)
\longrightarrow
\bigl[(H+\delta I)V^{(L)}\bigr](o),
\]
both in $L^2$.
\end{lemma}

\begin{proof}
It is easy to check, using Definition \ref{def:gamma-limit} and
induction on $L$,  that $\lim_{M\to\infty}b_{u\to v}^{(L,M)}=b_{u\to v}^{(L)}$
for every fixed $L$.
It follows that, for every $u$ with finitely many neighbors,
\begin{align*}
    V^{(L,M)}(u)
&=
\frac{t \sum_{v\sim u}
b_{v\to u}^{(L,M)}}{t^2+\deg(u)-1+\delta
-t\sum_{v\sim u}(\Gamma_v^{(u)})^{(M)}}
\\
&\longrightarrow
\frac{t \sum_{v\sim u}
b_{v\to u}^{(L)}}{t^2+\deg(u)-1+\delta
-t\sum_{v\sim u}(\Gamma_v^{(u)}) }
=
V^{(L)}(u).
\end{align*}
In particular, it holds almost
surely at $o$ and at each of its neighbors. Therefore,
\begin{align*}
\bigl[(H+\delta I)V^{(L,M)}\bigr](o)
&=
(\alpha+\deg(o))V^{(L,M)}(o)
-t\sum_{x\sim o}V^{(L,M)}(x)\\
&\longrightarrow
(\alpha+\deg(o))V^{(L)}(o)
-t\sum_{x\sim o}V^{(L)}(x)=
\bigl[(H+\delta I)V^{(L)}\bigr](o)
\end{align*}
almost surely.

By \Cref{lem:polynomial-bounded}, there is
a constant $C_L<\infty$, independent of $M$, such that
\[
|V^{(L,M)}(o)|
+
\left|
\bigl[(H+\delta I)V^{(L,M)}\bigr](o)
\right|
\leq C_L|B_{L+2}(T,o)|^{L+2}.
\]
Taking $M\to\infty$ in this inequality gives
\[
|V^{(L)}(o)|
+
\left|
\bigl[(H+\delta I)V^{(L)}\bigr](o)
\right|
\leq C_L|B_{L+2}(T,o)|^{L+2}
\qquad\text{almost surely}.
\]
The triangle inequality therefore gives
\begin{align*}
&|V^{(L,M)}(o)-V^{(L)}(o)|
+
\left|
\bigl[(H+\delta I)(V^{(L,M)}-V^{(L)})\bigr](o)
\right|
\leq 2C_L|B_{L+2}(T,o)|^{L+2},\\
&|V^{(L,M)}(o)-V^{(L)}(o)|^2 +
\left|
\bigl[(H+\delta I)(V^{(L,M)}-V^{(L)})\bigr](o)
\right|^2
\leq 4C_L^2|B_{L+2}(T,o)|^{2L+4}.
\end{align*}
Since $L$ is fixed, we have $\EE |B_{L+2}(T,o)|^{2L+4}<\infty$.
We know from the previous paragraph that both squared differences converge to zero almost surely.  Dominated convergence therefore gives
\[
\EE|V^{(L,M)}(o)-V^{(L)}(o)|^2
\longrightarrow0
\qquad \text{and}\qquad
\EE\left|
\bigl[(H+\delta I)(V^{(L,M)}-V^{(L)})\bigr](o)
\right|^2
\longrightarrow0.
\]

\end{proof}

\section{Finite convergence}

We now transfer the tree construction to the finite SBM. 
We will first truncate both recursions, obtaining fields determined
by finite neighborhoods. Local convergence then transfers
their moments and residuals to the finite graph. Then,
an eigenvalue count identifies the resulting approximate
eigenspaces with the actual spectral subspaces.

In this section, fix $\sqrt d\leq t<\mu_{r_+}^+$.
For notational convenience, relabel the positive informative
eigenvalues and their associated eigenfunctions as
$\mu_1,\dots,\mu_{r_+}$ and $\varphi_1,\dots,\varphi_{r_+}$. Let $G_n$ be an $n$-vertex SBM with vertex labeling $\sigma:V(G_n)\to [r]$, and let  $(T,o)$ be its limiting multi-type
Galton--Watson tree. 
The following local averaging lemma is standard; see, for example, \cite[Propositions~35--36]{bordenave2018nonbacktracking}. 

\begin{lemma}[Local averaging]\label{lem:convergence}Fix  $R\geq 0$.
Let $f$ be a real function on rooted graphs $(G,u)$ with vertex labels in $[r]$. Suppose $f$ only depends on $B_R(G,u)$ and that for some fixed $C,p<\infty$, we have
\[
|f(G,u)|\leq C|B_R(G,u)|^p\qquad \text{for all $(G,u)$}.
\]
Then we have
\[
\frac1n\sum_{u\in V(G_n)}f(G_n,u)
\longrightarrow \EE f(T,o)
\qquad\text{in }L^2.
\]
\end{lemma}

We previously fixed $\delta=\delta_i=\delta(\mu_i)$ and $X=X_i$. In the following proposition, we will add the subscript $i$ to emphasize the dependence on $i\in[r_+]$, and write $V_i$ for the field $V$ on $T$ as in Definition \ref{def:V}.

\begin{prop}\label{prop:finite-fields}
Let $\delta_i=\delta(\mu_i)$, $s_i^2\deq \EE[V_i(o)^2]>0$, and
$a_i\deq \EE[V_i(o)X_i(o)]$.
There are deterministic sequences $L(n),M(n)\to\infty$
such that, for
\[
V_{i,n}(u):=V_i^{(L(n),M(n))}(G_n,u),
\]
we have:
\begin{enumerate}
    \item $\frac1n\langle V_{i,n},V_{j,n}\rangle
\longrightarrow \mathbf 1_{\{i=j\}}s_i^2$ and $\frac1n\langle V_{i,n},X_j\rangle
\longrightarrow \mathbf 1_{\{i=j\}}a_i$ in $L^2$;
\item $\frac1n
\|(H(t,G_n)+\delta_i I)V_{i,n}\|^2
\longrightarrow0$;
\item The empirical measure
\[
\frac1n\sum_{u\in V(G_n)}
\delta_{(\sigma(u),V_{1,n}(u),\ldots,V_{r_+,n}(u))}
\]
converges weakly in probability to the law of
$(\sigma(o),V_1(o),\ldots,V_{r_+}(o))$.
\end{enumerate} 
\end{prop}

\begin{proof}

We choose bounded Lipschitz functions
$\Psi_m:[r]\times\RR^{r_+}\to[-1,1]$ such that, for
probability measures $(\nu_\ell)_{\ell\geq 1},\nu$ on
$[r]\times\RR^{r_+}$, $\nu_\ell\to\nu$ weakly if and only if $\int\Psi_m\,d\nu_\ell\longrightarrow\int\Psi_m\,d\nu$ for every $m\geq 1$. For example, let $(\Psi_m)_{m\ge1}$ enumerate the functions
\[
\Psi_{a,\alpha}(b,x)
\deq \mathbf1\{b=a\}\prod_{j=1}^{r_+}(\tanh x_j)^{\alpha_j},
\qquad
a\in[r],\quad \alpha\in\mathbb Z_{\ge0}^{r_+}.
\]
Each function is Lipschitz and takes values in $[-1,1]$. The linear span of these functions is an algebra of bounded
continuous functions that strongly separates points. Thus this
family determines weak convergence by
\cite[Chapter~3, Theorem~4.5(b)]{ethier1986markov}.

Write
\[
\eta_{n,L,M}\deq 
\frac1n\sum_{u\in V(G_n)}
\delta_{(\sigma(u),V_1^{(L,M)}(G_n,u),\ldots,
V_{r_+}^{(L,M)}(G_n,u))}
\]
and let $\eta$ be the law of
$(\sigma(o),V_1(o),\ldots,V_{r_+}(o))$.

Recall (e.g., from \cite[Lemma 27]{bordenave2018nonbacktracking}) that if $T_0$ is the underlying unlabelled tree, then we have $\EE(X_i(x)X_j(y)\mid T_0)=0$ for $i\neq j$.
Thus, by \Cref{lem:finite-field-approximation,lem:vertex-field-convergence} and
Cauchy--Schwarz, we have
\[
\lim_{L\to\infty}\lim_{M\to\infty}
\EE\left[V_i^{(L,M)}(o)V_j^{(L,M)}(o)\right]
=
\mathbf 1_{\{i=j\}}s_i^2,\qquad
\lim_{L\to\infty}\lim_{M\to\infty}
\EE\left[V_i^{(L,M)}(o)X_j(o)\right]
=
\mathbf 1_{\{i=j\}}a_i.
\]
Similarly, \Cref{lem:finite-field-approximation,lemma:residual}
give
\[
\lim_{L\to\infty}\lim_{M\to\infty}
\EE\left|
\bigl[(H(t,T)+\delta_i I)V_i^{(L,M)}\bigr](o)
\right|^2=0.
\]
Since each $\Psi_m$ is Lipschitz, \Cref{lem:vertex-field-convergence,lem:finite-field-approximation} also give
\begin{align*}
&\lim_{L\to\infty}\lim_{M\to\infty}
\EE\Psi_m\bigl(
\sigma(o),V_1^{(L,M)}(o),\ldots,V_{r_+}^{(L,M)}(o)
\bigr)=
\EE\Psi_m\bigl(\sigma(o),V_1(o),\ldots,V_{r_+}(o)\bigr)
=
\int\Psi_m\,d\eta.
\end{align*}
For fixed $L,M$, \Cref{lem:local,lem:polynomial-bounded}
allow us to apply \Cref{lem:convergence} to
\[
V_i^{(L,M)}(u)V_j^{(L,M)}(u),
\qquad
V_i^{(L,M)}(u)X_j(u),\qquad \left|
\bigl[(H(t,G)+\delta_i I)V_i^{(L,M)}\bigr](u)
\right|^2,
\]
and
\[
\Psi_m\bigl(
\sigma(u),V_1^{(L,M)}(u),\ldots,V_{r_+}^{(L,M)}(u)
\bigr).
\]
For any integer $k\geq 1$, we can therefore choose integers $L_k,M_k\geq k$
and $n_k$ such that for any $n\geq n_k$,
\begin{align}\label{eq:limitbounds}
\begin{split}
\max_{i,j\in[r_+]}
&\left\|
\frac1n\sum_{u\in V(G_n)}
V_i^{(L_k,M_k)}(G_n,u)V_j^{(L_k,M_k)}(G_n,u)
-\mathbf 1_{\{i=j\}}s_i^2
\right\|_{2}
\leq\frac1k,\\
\max_{i,j\in[r_+]}
&\left\|
\frac1n\sum_{u\in V(G_n)}
V_i^{(L_k,M_k)}(G_n,u)X_j(u)
-\mathbf 1_{\{i=j\}}a_i
\right\|_{2}
\leq\frac1k,\\
\max_{i\in[r_+]}
&\left\|
\frac1n\sum_{u\in V(G_n)}
\left|
\bigl[(H(t,G_n)+\delta_i I)V_i^{(L_k,M_k)}\bigr](u)
\right|^2
\right\|_{2}
\leq\frac1k,\\
\max_{1\leq m\leq k}
&\left\|
\int\Psi_m\,d\eta_{n,L_k,M_k}
-\int\Psi_m\,d\eta
\right\|_{2}
\leq\frac1k.
\end{split}
\end{align}
For every $n\geq n_1$, choose $k$ such that $n_k\leq n<n_{k+1}$, and set $(L(n),M(n))=(L_k,M_k)$. 
Then $L(n),M(n)\to\infty$, and the first three
inequalities above prove (1) and (2) in $L^2$.

For (3), set $\eta_n=\eta_{n,L(n),M(n)}$ and define
\[
d_*(\alpha,\beta)
\deq 
\sum_{m=1}^{\infty}2^{-m}
\left|\int\Psi_m\,d\alpha-\int\Psi_m\,d\beta\right|.
\]
By the choice of $(\Psi_m)$, convergence in $d_*$
is equivalent to weak convergence.
For $n_k\leq n<n_{k+1}$, the fourth inequality of \eqref{eq:limitbounds} and
$|\Psi_m|\leq1$ give
\[
\|d_*(\eta_n,\eta)\|_{L^2}
\leq
\sum_{m=1}^k2^{-m}\frac1k
+
2\sum_{m>k}2^{-m}
\leq
\frac1k+2^{1-k}
\longrightarrow0.
\]
Thus $d_*(\eta_n,\eta)\to0$ in probability, proving (3).
\end{proof}

It remains to identify these approximate eigenspaces
with the spectral subspaces of $H(t,G_n)$.

\begin{prop}\label{prop:eigenspace-approximation}
There exist orthonormal eigenvectors $v_{1,n},\ldots,v_{r_+,n}$
of $H(t,G_n)$, with corresponding eigenvalues
$-\delta_{1,n},\ldots,-\delta_{r_+,n}$, such that
\[
\delta_{i,n}\xrightarrow{\mathbb P}\delta_i
\qquad (1\le i\le r_+).
\]
For each distinct $\mu\in\{\mu_1,\ldots,\mu_{r_+}\}$, write
$J_\mu=\{i:\mu_i=\mu\}$. There is an orthogonal matrix
$O_{\mu,n}$ indexed by $J_\mu$ such that
\[
\max_{i\in J_\mu}
\left\|
\sum_{j\in J_\mu}(O_{\mu,n})_{ji}v_{j,n}
-\frac{V_{i,n}}{\|V_{i,n}\|_2}
\right\|_2
\xrightarrow{\mathbb P}0.
\]
In particular, if $\mu_i$ is simple, the sign of $v_{i,n}$
can be chosen so that
\[
\left\|v_{i,n}-\frac{V_{i,n}}{\|V_{i,n}\|_2}\right\|_2
\xrightarrow{\mathbb P}0.
\]
\end{prop}

\begin{proof}
Write $H_n=H(t,G_n)$, $m_\mu=|J_\mu|$, and
$\delta_\mu=\delta(\mu)$. All matrix norms below are operator
norms, and all limits are in probability.

Choose $\gamma>0$ so that the intervals
\[
I_\mu=(-\delta_\mu-\gamma,-\delta_\mu+\gamma)
\]
are pairwise disjoint and negative. Let $P_{\mu,n}$ be the
spectral projection of $H_n$ onto $I_\mu$, and let $U_{\mu,n}$
have columns $V_{i,n}/\|V_{i,n}\|_2$, $i\in J_\mu$.
Proposition~\ref{prop:finite-fields} gives
\[
U_{\mu,n}^{\mathsf T}U_{\mu,n}\longrightarrow I_{m_\mu},
\qquad
\|(H_n+\delta_\mu I)U_{\mu,n}\|\longrightarrow0.
\]
By the spectral theorem,
\[
\|(I-P_{\mu,n})U_{\mu,n}\|
\le \gamma^{-1}\|(H_n+\delta_\mu I)U_{\mu,n}\|
\longrightarrow0.
\]
Consequently,
\[
U_{\mu,n}^{\mathsf T}P_{\mu,n}U_{\mu,n}
\longrightarrow I_{m_\mu},
\]
so $\operatorname{rank}P_{\mu,n}\ge m_\mu$ with probability
tending to one.

Choose $\epsilon>0$ so that every $I_\mu$ lies below
$-\epsilon$. The negative-outlier bound of \cite[Corollary~3.4]{dong2026bethe} gives, with probability
tending to one,
\[
\sum_\mu\operatorname{rank}P_{\mu,n}
\le \#\{\lambda\in\sigma(H_n):\lambda<-\epsilon\}
\le r_+.
\]
Since $\sum_\mu m_\mu=r_+$, we obtain
$\operatorname{rank}P_{\mu,n}=m_\mu$ for every $\mu$.

On this event, set
\[
W_{\mu,n}
=P_{\mu,n}U_{\mu,n}
\bigl(U_{\mu,n}^{\mathsf T}P_{\mu,n}U_{\mu,n}\bigr)^{-1/2}.
\]
Its columns form an orthonormal basis of
the range of $P_{\mu,n}$, and
\[
\|W_{\mu,n}-U_{\mu,n}\|\longrightarrow0.
\]
Moreover, since $P_{\mu,n}$ commutes with $H_n$,
\[
(H_n+\delta_\mu I)W_{\mu,n}
=
P_{\mu,n}(H_n+\delta_\mu I)U_{\mu,n}
\bigl(U_{\mu,n}^{\mathsf T}P_{\mu,n}U_{\mu,n}\bigr)^{-1/2}
\longrightarrow0.
\]
Thus all eigenvalues in $I_\mu$ converge to $-\delta_\mu$.

Choose an orthonormal eigenbasis $(v_{i,n})_{i\in J_\mu}$
of the range of $P_{\mu,n}$. An orthogonal change of
basis takes these vectors to the columns of $W_{\mu,n}$,
giving the asserted approximation. The eigenvectors from
distinct intervals are orthogonal. When $m_\mu=1$, the
change of basis is simply a choice of sign.
\end{proof}
\begin{proof}[Proof of \Cref{thm:eigenvectors}]
Take $t=\sqrt d$. Proposition~\ref{prop:delta} gives the
deterministic strictly increasing function $\delta$.
Proposition~\ref{prop:eigenspace-approximation} gives the
outliers associated with each positive informative eigenvalue,
with the correct multiplicities. The outlier count used in its
proof shows that these are the first $r_+$ eigenvalues of
$H(t,G_n)$. Their ordering follows from the strict increase
of $\delta$.

For a negative planted eigenvalue $\mu<-t$, take
$\delta=\delta(|\mu|)$. In the boundary-message and vertex-field
constructions, use $-\Gamma_\delta$ in place of $\Gamma_\delta$
and $-t$ in place of $t$. The denominator in the vertex-field
formula is unchanged. Since $\rho=\mu/d<0$, the conditional
mean is propagated by
\[
\rho(-\cL_\delta)=|\rho|\cL_\delta.
\]
The same convergence estimates apply, and the resulting field
satisfies
\[
(H(-t,T)+\delta I)V=0.
\]
The local approximation, finite convergence, and corresponding
negative-outlier count therefore give the asserted eigenvalue
limits for $H(-t,G_n)$.

It remains to verify the overlap assertion. For either sign,
write $\delta_i=\delta(|\mu_i|)$, and let $\Gamma_i$ and $h_i$
be the positive cavity weight and transfer eigenfunction at
this shift. If $K$ is the root degree and $T_1,\ldots,T_K$
are its forward subtrees, taking conditional expectations
in the vertex-field formula gives
\[
a_i
=\frac{t|\mu_i|}{d}\,
\EE\left[
\frac{\sum_{j=1}^K h_i(T_j)}
{t^2+K-1+\delta_i-t\sum_{j=1}^K\Gamma_i(T_j)}
\right]>0.
\]
The denominator is positive, and the numerator is positive
whenever a forward subtree survives, by
Proposition~\ref{prop:kappa}. Thus $a_i>0$, and
Cauchy--Schwarz also gives $s_i\ge a_i>0$.

Fix a distinct informative eigenvalue $\mu$ and write
$J_\mu=\{i:\mu_i=\mu\}$. Let $(w_{i,n})_{i\in J_\mu}$
be the orthonormal basis of its outlier subspace constructed
in Proposition~\ref{prop:eigenspace-approximation}.
By Proposition~\ref{prop:finite-fields},
\[
\left(
\frac{\langle w_{i,n},X_j\rangle}{\sqrt n}
\right)_{i,j\in J_\mu}
\xrightarrow{\mathbb P}
\operatorname{diag}(a_i/s_i:i\in J_\mu).
\]
Consequently, with probability tending to one, the smallest
singular value of this matrix is at least
\[
c_\mu=\frac12\min_{i\in J_\mu}\frac{a_i}{s_i}>0.
\]
Since the vectors $X_i/\sqrt n$ are orthonormal, this is
precisely the claimed lower bound on the projection of
every planted vector onto the outlier subspace.
When $\mu_i$ is simple, the preceding convergence gives
\eqref{eq:eigenvector-overlap} with $q_{\mu_i}=a_i/s_i>0$.
\end{proof}
\section{Weak recovery}\label{sec:recovery}
Given our convergence to nontrivial eigenvectors, we use standard arguments to perform weak recovery.
\begin{proof}[Proof of \Cref{thm:weak-recovery}]
Let $S:=\sum_{\substack{i\geq 2\\\mu_i>\sqrt{d}}}c_{\mu_i}^2$. If $S=0$ we are done.
Suppose $S>0$, and let
\[
m=\big|\{i\ge2:|\mu_i|>\sqrt d\}\big|.
\]
Let $E_n$ be the sum of all informative outlier subspaces
of $H(\sqrt d,G_n)$ and $H(-\sqrt d,G_n)$.
Choose an orthonormal basis $U_n\in\RR^{n\times m}$ of
$E_n\cap\mathbf1^\perp$, and write
\[
x_u\deq \sqrt n\,U_n(u,\cdot)^{\mathsf T}.
\]
The joint limits in Propositions~\ref{prop:finite-fields}
and~\ref{prop:eigenspace-approximation} give
\[
\sum_{\substack{i\ge2\\|\mu_i|>\sqrt d}}
\left\|U_n^{\mathsf T}\frac{X_i}{\sqrt n}\right\|_2^2
\ge S-o_{\mathbb P}(1).
\]
Indeed, the combined outlier basis has a diagonal limiting
correlation matrix with the planted directions.
Intersecting with $\mathbf1^\perp$ removes only the constant
planted direction asymptotically, including when $d$ has
multiplicity greater than one.

Put
\[
L\deq \sqrt{\pi_{\min}S},
\qquad
M\deq \frac{4\sqrt2(r-1)}{L},
\qquad
f_u\deq \frac{x_u}{\max\{M,\|x_u\|_2\}}.
\]
Fix any matrix $R\in\RR^{r\times m}$ satisfying
$R^{\mathsf T}R=I_m$ and $R^{\mathsf T}\mathbf1=0$.
Independently over vertices, sample
\[
\mathbb P(\widehat\sigma_n(u)=b\mid G_n)
=\frac{1+(Rf_u)_b}{r},
\qquad b\in[r].
\]
These are valid probabilities because $\|f_u\|_2\le1$.

To bound the overlap, define
\[
C_n=\frac1n\sum_u(e_{\sigma(u)}-\pi)x_u^{\mathsf T},
\qquad
C_{f,n}=\frac1n\sum_u(e_{\sigma(u)}-\pi)f_u^{\mathsf T},
\]
where $e_a$ is the $a$th coordinate vector.
Orthonormality of the planted eigenfunctions gives
\[
\|C_n\|_{\mathrm F}^2
\ge
\pi_{\min}
\sum_{\substack{i\ge2\\|\mu_i|>\sqrt d}}
\left\|U_n^{\mathsf T}\frac{X_i}{\sqrt n}\right\|_2^2.
\]
Thus $\|C_n\|_{\mathrm F}\ge L/2$ with probability
tending to one. Since
$n^{-1}\sum_u\|x_u\|_2^2=m$,
\[
\|MC_{f,n}-C_n\|_{\mathrm F}
\le
\frac{\sqrt2}{Mn}\sum_u\|x_u\|_2^2
\le \frac L4.
\]
Consequently,
\[
\|C_{f,n}\|_{\mathrm F}\ge\frac{L}{4M}
\]
with probability tending to one.

Condition on $G_n$ and $\sigma$. The expected centered
confusion matrix is
\[
D_n\deq \frac1r C_{f,n}R^{\mathsf T}.
\]
It has zero row and column sums, and
$\|D_n\|_{\mathrm F}=\|C_{f,n}\|_{\mathrm F}/r$.
For any such nonzero matrix $D$,
\[
\max_{\tau\in S_r}\sum_bD_{\tau(b),b}
\ge\frac{\|D\|_{\mathrm F}}r.
\]
Indeed,
$\mathbf1\mathbf1^{\mathsf T}/r+D/(r\|D\|_{\mathrm F})$
is doubly stochastic, so the inequality follows from
Birkhoff's theorem.

It follows that some permutation has conditional expected
centered overlap at least
\[
\frac{L}{4Mr^2}
=\frac{\pi_{\min}S}{16\sqrt2\,r^2(r-1)}.
\]
Conditional independence and Hoeffding's inequality now give
the claimed bound. When $S=0$, the conclusion is immediate,
since centered overlap is always nonnegative.
\end{proof}
 
\bibliographystyle{amsplain}
\bibliography{ref}

\end{document}